\documentclass[10pt, reqno]{amsart}

\usepackage{amsmath, amsthm, amssymb}
\usepackage{fancyhdr}
\usepackage{array}
\usepackage{calc}
\usepackage{mathrsfs}
\usepackage{comment}
\usepackage{tikz-cd}
\usepackage[colorlinks=true,urlcolor=blue,citecolor=blue,linkcolor=blue]{hyperref}
\usepackage[colorlinks=true,urlcolor=blue,citecolor=blue,linkcolor=blue]{hyperref}
\usepackage{setspace}

\newtheorem{theorem}{Theorem}[section]
\newtheorem{corollary}[theorem]{Corollary}
\newtheorem{lemma}[theorem]{Lemma}
\newtheorem{proposition}[theorem]{Proposition}

\theoremstyle{definition}
\newtheorem{definition}[theorem]{Definition}
\newtheorem{convention}[theorem]{Convention}

\theoremstyle{remark}
\newtheorem{remark}[theorem]{Remark}
\newtheorem{example}[theorem]{Example}

\newenvironment{proof*}{\par\noindent{\it Proof.}}{}
\numberwithin{equation}{section}

\def\Zero{\mathbf{0}}
\def\One{\mathbf{1}}
\def\Infinity{\mathbf{\infty\llap{$\infty$\kern.1pt}}}
\def\funfld{\kappa}
\def\sd{\mathop{\#}} 
\def\ad{\mathrm{ad}}
\def\Gm{\mathbb{G}_m}

\def\cC{\mathcal{C}}
\def\cL{\mathcal{L}}
\def\cP{\mathcal{P}}

\def\OO{\mathcal{O}}

\def\AA{\mathbb{A}}
\let\Aff\AA
\def\FF{\mathbb{F}}
\def\PP{\mathbb{P}}
\def\QQ{\mathbb{Q}}
\def\ZZ{\mathbb{Z}}

\def\mm{\mathfrak{m}}
\def\pp{\mathfrak{p}}
\def\qq{\mathfrak{q}}

\def\Ann{\operatorname{Ann}}
\def\Proj{\operatorname{Proj}}
\def\PGL{\operatorname{PGL}}
\def\GL{\operatorname{GL}}
\def\Hom{\operatorname{Hom}}
\def\Pic{\operatorname{Pic}}
\def\ord{\operatorname{ord}}
\def\im{\operatorname{im}}
\def\Spec{\operatorname{Spec}}
\def\Sym{\operatorname{Sym}}
\def\Aut{\operatorname{Aut}}

\let\spec\Spec
\let\isom\cong

\def\setof#1{{\let\suchthat\colon\let\st\suchthat\{#1\}}}
\def\mat#1{\left(\begin{smallmatrix} #1 \end{smallmatrix}\right)}

\title{Cross-Ratios of Scheme-Valued Points}
\author{Xander Faber}
\address{XF: Center for Computing Sciences \\
Institute for Defense Analyses \\
Bowie, MD}
\email{awfaber@super.org}
  
\author{Keith Pardue}
\address{KP: National Security Agency, Fort Meade, MD}
\author{David Zelinsky}
\address{DZ: National Security Agency, Fort Meade, MD}
\begin{document}

\begin{abstract}
The classical theory of the cross-ratio is a beautiful case study of the moduli
of ordered points of the projective line and of invariants of the action of
$\PGL_2$. We generalize the theory of the cross-ratio to the setting of
$S$-valued points for an \textit{arbitrary} scheme $S$. To accomplish this goal,
we provide a comprehensive and computationally focused treatment of
automorphisms of projective space over $S$, of equalizers in the category of
schemes, and of vanishing loci of sections of line bundles. Most of these ideas
exist in the literature, though not with the level of detail or generality that
we require. After introducing the notion of a ``strongly distinct'' pair of
morphisms, we define the cross-ratio of 4-tuples of pairwise strongly distinct
$S$-valued points of the projective line --- which is valued in the units of the
ring of global functions on the scheme $S$ --- and show that it enjoys all of
the familiar properties of the cross-ratio.
\end{abstract}

\maketitle\thispagestyle{fancy}


\section{Introduction}

On the projective line over a field $k$, the action of $\PGL_2(k)$ is triply
transitive: Any triple of distinct points can be mapped to any other triple.
However, the action is not 4-transitive.  Indeed, given a 4-tuple of distinct
points, $a,b,c,d$, its orbit under the $\PGL_2(k)$ action is determined by the
cross-ratio, $\chi(a,b,c,d):=\frac{(a-c)(b-d)}{(a-d)(b-c)}$.  That is, if
$a',b',c',d'$ is another 4-tuple of distinct points, there is an element
$M\in\PGL_2(k)$ such that $a'=M.a$, $b'=M.b$, $c'=M.c$, and $d'=M.d$, if and
only if $\chi(a,b,c,d) = \chi(a',b',c',d')$.

The purpose of this note is to show how the above scenario can be generalized to
the case of the projective line $\PP^1_S$ over an arbitrary scheme $S$.  There
are several obstacles that must be dealt with in this quest.  First, we must
understand what takes the place of $\PGL_2(k)$.  That is, we must identify the
automorphism group of $\PP^1_S$.  This is well known, but as we will need to
understand the action in detail, we give an account of this story in
\S\ref{sec:Automorphisms}.  Second, we need to be clear what we mean by
``point''.  For us, this will mean an $S$-valued point; \emph{i.e.}, a section
of the structure map $\PP^1_S\to{}S$.

Finally, and most crucially, we run into the problem that, in general, the
action of $\Aut_S(\PP^1_S)$ on $S$-points is not even \emph{singly} transitive!
However, we can view $\PP^1_S$ as a family of $\PP^1$'s, varying over
the base $S$, and through that lens we see that the right generalization of ``4
distinct points of $\PP^1(k)$'' should be ``4 sections that are distinct in
each fiber of $\PP^1_S$''. In \S\ref{sec:StrongDistinctness} we introduce the
notion of \emph{strong~distinctness} for a pair of $S$-points of $\PP^1_S$, or
more generally for a pair of scheme morphisms $f, g \colon S \to X$.  Morally,
$f$ and $g$ are strongly distinct if they take distinct values at all points of
$S$. This slogan turns out to be equivalent to the definition in some cases,
such as when $S$ is a variety over an algebraically closed field or when $X$ is
an $S$-scheme and $f,g$ are $S$-morphisms (Corollary~\ref{cor:moral_sd}).

Having overcome all of these obstacles, our main result
(Theorem~\ref{thm:3-transitivity}) is that the automorphism group of $\PP^1_S$
acts simply transitively on triples of pairwise strongly distinct
$S$-points.  Moreover, there is a natural generalization of the cross-ratio to
4-tuples of pairwise strongly distinct points
(Definition~\ref{def:cross-ratio}), such that two of these 4-tuples are in the
same $\Aut_S(\PP^1_S)$-orbit if and only if their cross-ratios are equal
(Theorem~\ref{thm:witness}).

The definition of strong distinctness amounts to a condition on the equalizer of
the two morphisms.  To deal with this, we discuss equalizers
in \S\ref{sec:Equalizers}.  Finally, to get a useful criterion for strong
distinctness (the Determinant Criterion, Corollary~\ref{cor:sd-criterion}), we
need some information about the zero locus of a section of a line bundle, which
we discuss in \S\ref{sec:ZeroSets}.

\emph{Some notation and terminology.}  As usual, we write $\OO_X$ for the
structure sheaf of a scheme $X$, and $\Gamma(X,\OO_X)$ for its ring of global
sections.  We also write $\kappa(X)$ for the function field of $X$, when $X$ is
integral.  In particular, if $x\in X$ is a point, then $\kappa(x)$ is its
residue field.  The term \textbf{vector bundle} will refer to a locally free sheaf
of finite rank.  A \textbf{line bundle} is a vector bundle of rank~1.


\section{Automorphisms of Projective Space}
\label{sec:Automorphisms}

Fix an integer $n \geq 2$ for the duration of this section. Our present goal is
to give a complete description of the automorphisms of $\PP^{n-1}_S$ for a
general scheme $S$. This is discussed briefly in \cite[pp.19-21]{GIT} for
noetherian schemes (possibly requiring that $S$ be a variety over an
algebraically closed field) and in \cite[\S6.2]{Hida} for connected affine
schemes. But we want to avoid placing any restrictions on the scheme $S$. While
many of the results stated here are well known, we do not believe there is a
complete account in the literature without some restrictions placed on $S$.  For
that reason, we give an exposition of the theory for the benefit of the reader.

Over a field $k$, it is a staple in every algebraic geometry course to prove
that the automorphism group of $\PP^{n-1}_k$ is isomorphic to
$\GL_n(k)/k^\times=\PGL_n(k)$.  If we wish to replace $k$ with a more general
base scheme $S$, we encounter a problem. The functor $\Spec
A\rightsquigarrow\GL_n(A)/A^\times$ from commutative rings to groups is not a
sheaf in the Zariski topology.  Put another way, the functor
$S\rightsquigarrow\GL_n(S)\!/\Gm(S)$ is not representable, whereas the functor\footnote{
  If $S \to T$ is a morphism of schemes, then base-extending an automorphism of
  $\PP^{n-1}_T$ to $S$ yields an automorphism of $\PP^{n-1}_S$.} taking $S$ to
$\Aut_S(\PP^{n-1}_S)$, the group of $S$-automorphisms of
$\PP^{n-1}_\ZZ\times{}S$, \emph{is} representable by a certain group scheme over
$\ZZ$.\footnote{\label{fn:Grothendieck}Grothendieck raises this issue in
  \cite[\S8]{Grothendieck_Preschemas_Quotients}, where he asserts that the
  ``group quotient'' $\GL_n / \GL_1$ exists, and represents the functor $S
  \mapsto \Aut_S(\PP^{n-1}_S)$, ``du moins pour $S$ noeth\'erien.''}  We will
reserve the name $\PGL_n$ to refer to that group scheme.  Then for any scheme
$S$, $\PGL_n(S)$ denotes the group of morphisms from $S$ to $\PGL_n$.  Thus we
have three functors from $\mathbf{Sch}$ to $\mathbf{Grp}$:
\begin{equation*}
\begin{aligned}
  S &\rightsquigarrow \GL_n(R) / R^\times \\
    & \quad \text{with } R = \Gamma(S,\OO_S) \text{ the ring of global functions on } S, \\
  S &\rightsquigarrow \PGL_n(S), \\
  S &\rightsquigarrow \Aut_S(\PP^{n-1}_S).
\end{aligned}
\end{equation*}
When $S = \Spec R$ is the spectrum of a local ring, all three of these functors
give the same group.  More generally, we will see that the first functor is a
subgroup of the second.  We will also show that the last two are naturally
isomorphic, which is what it means for $\Aut_{\bullet}(\PP^{n-1}_{\bullet})$ to be
represented by $\PGL_n$.  Along the way, we will introduce a fourth functor
$S \rightsquigarrow G_n(S)$ that is useful for explicitly writing down elements of
$\PGL_n(S)$.  Finally, we will close this section with an example that
illustrates the difference between $\GL_2(R) / R^\times$ and $\PGL_2(R)$ for a
general ring $R$.

\begin{remark}
  \label{rem:brauer1}
  A fifth functor ought to be considered in the above story.  Let $M_n(R)$ be
  the algebra of $n \times n$ matrices with coefficients in $R =
  \Gamma(S,\OO_S)$. The Skolem-Noether theorem \cite[IV.1.4]{Milne_EC} shows
  that when $R$ is local, every automorphism of $M_n(R)$ is realized as
  conjugation by an element of $\GL_n(R)$. Consequently, the functor $S
  \rightsquigarrow \Aut(M_n(R))$ is naturally isomorphic to $S \rightsquigarrow
  \GL_n(R) / R^\times$. We will not need this description, though it does
  provide a different viewpoint for Corollary~\ref{cor:exact_seq} below. See
  Remark~\ref{rem:brauer2}.
\end{remark}


\newpage
\subsection{Properties of PGL}

\begin{definition}
Set $N = n^2 - 1$. Define the group scheme $\PGL_n$ as the open
subscheme of
\[
\PP^N=\Proj\,\ZZ\big\lbrack\setof{x_{i,j}\st{}1\le{}i,j\le{}n}\big\rbrack
\] 
obtained by removing the zero locus of the determinant form
$\Delta=\det((x_{i,j}))\in{}\Gamma(\PP^N,\OO_{\PP^N}(n))$. For a scheme $S$, we
define $\PGL_{n}(S)$ to be the set of $S$-points of the scheme $\PGL_{n}$. For a
ring $R$, we write $\PGL_n(R)$ instead of $\PGL_n(\Spec R)$.
\end{definition}

\goodbreak

The usual matrix multiplication formulas define a rational map
$\PP^N \times \PP^N \dashrightarrow \PP^N$, and the fact that the determinant is
multiplicative shows that away from $\Delta=0$, this is a morphism
\hbox{$m \colon \PGL_{n} \times \PGL_{n} \to \PGL_{n}$}.

\begin{proposition}
$\PGL_n$ is a group scheme with multiplication map $m$.
\end{proposition}

\begin{proof}
Clearly $m$ is associative, since the same is true of matrix multiplication.
$\PP^N$ has the affine open subset
$D(x_{1,1})=\spec\ZZ[\setof{x_{i,j}/x_{1,1}\st (i,j)\ne(1,1)}]$, and the map
$\spec\ZZ\to D(x_{1,1})$ given by $x_{i,j}/x_{1,1}\mapsto\delta_{i,j}$
(Kronecker delta) defines an element of $\PGL_n(\ZZ)$ that serves as the group
identity.  Finally, to define the inversion morphism, let $M=(x_{i,j})$ and
consider the usual adjugate formula for the inverse of a matrix,
$M^{-1}=\Delta^{-1}A$, where $A$ is the adjugate matrix to $M$. The entries of
$A$ are homogeneous polynomials of degree $n-1$ in $\setof{x_{i,j}}$, so $A$
defines a rational map $i:\PP^N\dashrightarrow\PP^N$.  From the formula, and
properties of $\det$, we see that $\det(A) = \det(M)^{n-1} = \Delta^{n-1}$, so
that $i$ induces a morphism from $\PGL_n$ to itself.  Since multiplication by
$\Delta^{-1}$ induces the identity morphism on $\PGL_n$, we find that $i$ serves
as the inverse morphism for the group law $m$. That is, $\PGL_n$ is a group
scheme.
\end{proof}

\begin{definition}
  \label{def:G_n}
For any scheme $S$, let $G_n(S)$ be the set of equivalence
classes
\footnote{More precisely, we choose a set of representatives, one from each
  equivalence class. Such a set of representatives exists because line bundles
  are locally finitely presented. But, the equivalence classes themselves are
  proper classes, not eligible for set membership without expanding the
  universe. We do not belabor the difference any further in this paper.} of
pairs $(L;M)$, where:
\begin{itemize}
\item $L$ is a line bundle on $S$ such that $L^{\otimes n}\cong\OO_S$;
\item $M= (m_{i,j})$ is an $n \times n$ matrix of global sections
  of $L$ such that the determinant
  \[
  \det(M) := \sum_{\sigma \in S_n} (-1)^{\mathrm{sgn}(\sigma)} \ m_{1,
    \sigma(1)} \otimes \cdots \otimes m_{n, \sigma(n)}
  \]
  is a nowhere-vanishing section of $L^{\otimes{}n}$;
\item $(L;M)\sim (L';M')$ if there exists an
isomorphism $\varphi:L\to{}L'$ such that $\varphi(M)=M'$, with $\varphi(M)$ computed by
applying $\varphi$ to each entry of $M$.
\end{itemize}
We endow $G_n(S)$ with a group structure via $(L;M)\cdot (L';M') = (L \otimes L'; M
\star M')$, where $\star$ denotes ``tensor multiplication''. That is, $(M \star
M')_{i,j} = \sum_k m_{i,k} \otimes m_{k,j}' \in \Gamma(S, L \otimes L')$. 
\end{definition}

\begin{remark}
The identity element of $G_n(S)$ is $(\OO_S; I)$, where $I_{i,j} = \delta_{i,j}$.
Given $(L;M) \in G_n(S)$, write $\ad(M)$ for the ``tensor adjugate'' of $M$. That
is, $\ad(M)_{i,j} = (-1)^{i+j}\det(M^{j,i})$, where $M^{j,i}$ is the matrix of
global sections of $L$ obtained by removing the $j$-th row and $i$-th column
from $M$. Then one verifies that $(L^{\otimes (n-1)}; \ad(M))$ is the inverse of
$(L; M)$ in the group $G_n(S)$.
\end{remark}

\begin{proposition}
\label{prop:simplePGLn}
The functors $S \rightsquigarrow \PGL_n(S)$ and $S \mapsto G_n(S)$
are naturally isomorphic. 
\end{proposition}

\begin{proof}
Fix a scheme $S$ and set $N = n^2 - 1$. By definition,
$\PGL_n(S)=\Hom(S,\PGL_n)$ is the subset of $\PP^N(S)$ comprising morphisms
$S\to\PP^N$ with image disjoint from the zero locus of $\Delta$.  Using
\cite[II.7.1]{Hartshorne_Bible}, we see that up to a natural notion of
equivalence, such maps are in bijective correspondence with data
$(L;(s_{i,j})_{1\le{}i,j\le{}n})$ where $L$ is a line bundle on $S$, and
$s_{i,j}$ are global sections of $L$ that generate $L$ and that satisfy
$\Delta(s_{1,1},\dots,s_{n,n})\notin\mathfrak{m}_xL^{\otimes{}n}_x$ locally at
every point $x\in{}S$.  This last condition is the same as saying
$\det((s_{i,j}))$ generates $L^{\otimes{}n}$; that is, $(L;M)$ represents an
element of $G_n(S)$, where $M$ is the matrix with entries $M_{i,j}=s_{i,j}$.
Moreover, the criterion for data $(L;M)$ and $(L';M')$ to define the same point
of $\PP^N(S)$, hence of $\PGL_n(S)$, is exactly the equivalence relation
defining $G_n(S)$. Thus, we get a bijection of sets $\PGL_n(S)
\stackrel{\sim}{\to} G_n(S)$. It is an isomorphism of groups because the group
operation on each side is defined by matrix multiplication.

For any morphism of schemes $T \to S$, line bundles and sections pull back, so
naturality of the isomorphisms $\PGL_n(\bullet) \stackrel{\sim}{\to}
G_n(\bullet)$ follows immediately.
\end{proof}

The functor $G_n$ is equivalent to $\PGL_n$, but the former is much easier to
work with. Henceforth, we will identify them without comment.  In particular, we
will treat elements of $\PGL_n(S)$ as being represented by pairs
$(L;M)\in{}G_n(S)$.

\begin{corollary}
  \label{cor:exact_seq}
  For any scheme $S$ with ring of global functions $R = \Gamma(S,\OO_S)$, we
  have a natural exact sequence of groups
  \begin{equation}
    \label{eq:exact}
  1 \to R^\times \to \GL_n(R) \to \PGL_n(S) \to \Pic(S)[n].
  \end{equation}
  In particular, if $\Pic(S)[n] = 0$, then $\PGL_n(S) \cong \GL_n(R) / R^\times$. 
\end{corollary}

\begin{proof}
Proposition~\ref{prop:simplePGLn} shows that $\PGL_n \cong G_n$ as functors, so
we identify them without further comment. 

The map $R^\times \to \GL_n(R)$ takes $r$ to the scalar matrix $rI$, where $I$
is the $n\times n$ identity matrix. The map $\GL_n(R) \to \PGL_n(S)$ is given by
$M \mapsto (\OO_S; M)$. The map $\PGL_n(S) \to \Pic(S)[n]$ is given by $(L; M)
\mapsto L$. Since $L^{\otimes{n}} \cong \OO_S$ for $(L;M) \in G_n(S)$, we find
that $L \in \Pic(S)[n]$. Naturality of the sequence \eqref{eq:exact} in $S$ is
immediate from the definition of the maps.

It is obvious that $R^\times \to \GL_n(R)$ is injective. Let us show that the
sequence \eqref{eq:exact} is exact at $\GL_n(R)$. First, the composition
$R^\times \to \GL_n(R) \to \PGL_n(S)$ is trivial because it carries $r$ to
$(\OO_S; rI) \sim (\OO_S; I)$. Next, if the matrix $M$ lies in the kernel of
$\GL_n(R) \to \PGL_n(S)$, then there is an $\OO_S$-module automorphism
$\varphi \colon \OO_S \to \OO_S$ such that $\varphi(I) = M$. Since $\varphi$
must be multiplication by some $r \in R^\times$, we see that $M = rI$ is a
scalar matrix. So $M$ is in the image of $R^\times \to \GL_n(R)$, as desired.

Now we show that the sequence \eqref{eq:exact} is exact at $\PGL_n(S)$. The
composition
\begin{equation*}
\GL_n(R) \to \PGL_n(S) \to \Pic(S)[n]
\end{equation*}
is trivial because it
carries a matrix $M$ to the trivial bundle $\OO_S$. Now suppose that $(L; M) \in
\PGL_n(S)$ has trivial image in the Picard group. Then $L \cong \OO_S$ via some
isomorphism $\varphi \colon L \to \OO_S$. It follows that $(L; M) \sim (\OO_S;
\varphi(M))$ is the image of $\varphi(M) \in \GL_n(R)$.
\end{proof}

\begin{remark}
  \label{rem:brauer2} An alternate proof of Corollary~\ref{cor:exact_seq}
    can be given using sheaf cohomology. By the Skolem-Noether theorem, the
    sequence of sheaves
    \[
      1 \to \mathbb{G}_m \to \GL_n \to \PGL_n \to 1
      \]
    is exact for the Zariski topology. Passing to the long exact sequence on
    sheaf cohomology gives
    \[
      1 \to R^\times \to \GL_n(R) \to \PGL_n(R) \to H^1(S, \mathbb{G}_m) \to H^1(S, \GL_n) \to \cdots
      \]
    It is well known that $H^1(S, \mathbb{G}_m) \cong \Pic(S)$, and that $H^1(S,
    \GL_n)$ is isomorphic to the \textit{set} of vector bundles on $S$ of rank
    $n$. The final map is given by $L \mapsto L^{\oplus n}$. One can use the
    determinant bundle to show that the kernel of this map lies inside the
    $n$-torsion of $\Pic(S)$.
\end{remark}

\begin{corollary}
  \label{cor:local_case}
If $R$ is a commutative ring, then $\GL_n(R)/R^\times$ is a normal subgroup of
$\PGL_n(R)$.  If $R$ is a local ring, then $\PGL_n(R) = \GL_n(R) / R^\times$.
\end{corollary}

\begin{proof}
If $S=\spec{R}$, then $R = \Gamma(S,\OO_S)$, and the first result follows from
the preceding corollary.  If $R$ is local, then $\Pic(S)=\Pic(R)=0$, so
$\GL_n(R)\to\PGL_n(R)$ is surjective with kernel $R^\times$.
\end{proof}

At the end of this section, we will exhibit an example of a ring $R$ and an
element of $\PGL_n(R)$ with nontrivial image in $\Pic(R)[n]$.  However, the map
$\PGL_n(S)\to\Pic(S)[n]$ is not always surjective.  In fact we have

\begin{corollary}
Let $S$ be a complete integral variety over an algebraically closed field
$k$. Then
\begin{enumerate}
\item $\PGL_n(S) = \PGL_n(k)$; and
\item the image of $\PGL_n(S)$ in $\Pic(S)[n]$ is zero.
\end{enumerate}
\end{corollary}

\begin{proof}
Let $(L;M)$ be an element of $G_n(S)=\PGL_n(S)$.  Since $\det(M)$ generates
$L^{\otimes{}n}$, some entry of $M$ is a nonzero global section of $L$. Suppose
$s$ is such a section. Then $s^{\otimes{}n}$ is a nonzero global section of
$L^{\otimes{}n}\isom\OO_S$.  As $S$ is complete and integral, we have
$\Gamma(S,\OO_S)=k$, and every nonzero section of $\OO_S$ is \emph{nowhere}
vanishing. But because $S$ is reduced, $s$ has a nonzero germ wherever $s^{\otimes{}n}$ does,
so this means that $s$ is also nowhere vanishing. Hence $L\isom\OO_S$, and
$(L;M)$ maps to zero in $\Pic(S)[n]$.  This proves (2).  Now (1) follows from
(2) together with the preceding two corollaries.
\end{proof}


\subsection{Automorphisms and PGL}

Now we turn to the proof that the functors
\begin{equation*}
S \rightsquigarrow \PGL_n(S)\quad\text{and}\quad S \rightsquigarrow\Aut(\PP^{n-1}_S)
\end{equation*}
are naturally isomorphic.  Since maps to projective space are
associated with line bundles, it stands to reason that computing
$\Aut_S(\PP^{n-1}_S)$ involves understanding line bundles on $\PP^{n-1}_S$.  The
following is the key lemma.  It is considered well known, but it seems to be
difficult to find a complete proof in the literature, without extra hypotheses
on $S$.%
\footnote{For example, the result is stated in \cite[pp.19-21]{GIT}, but at a
crucial point they refer to Mumford's ``Lectures on Curves on an Algebraic
Surface'' \cite{Mumford:Curves_on_Surfaces}.  Unfortunately, in that work one of
Mumford's standing assumptions is that all schemes are of finite type over an
algebraically closed field. See also Footnote~\ref{fn:Grothendieck}.}
We will sketch an approach to a proof, with references to all the relevant
ingredients.  We are grateful to Bjorn Poonen for providing us with this road
map. (See \cite[Ex.~1.6.3]{Conrad_RGS} for a sketch of a different argument in
the local case.)


\begin{lemma}
\label{lem:Pic-PP^n_S}
Let $S$ be a connected scheme, $\pi_S:\PP^n_S\to S$ the projection, and $\cL$
any line bundle on $\PP^n_S$.  Then there is a unique integer $d\in\ZZ$, and a
line bundle $L$ on $S$, unique up to isomorphism, such that
$\cL\cong \OO_{\PP^n_S}(d)\otimes\pi_S^*(L)$.
\end{lemma}

\begin{proof}[Sketch of proof]
If $Z$ is any locally noetherian scheme and $X$ is a flat and projective
$Z$-scheme with geometrically integral fibers, a result of
Grothendieck \cite[no.232,Thm. 3.1]{FGA} says that a certain functor from
$Z$-schemes to groups, called the \emph{relative Picard functor}, $\Pic_{X/Z}$,
is representable by a group scheme.  In general, $\Pic_{X/Z}$ is defined by a
sheafification process applied to the functor $S\rightsquigarrow\Pic(X\times_Z{}S)$.  But
under certain hypotheses, $\Pic_{X/Z}$ is isomorphic to the functor
\begin{equation*}
S\rightsquigarrow\Pic(X\times_Z{}S)/\pi_S^*\Pic(S),
\end{equation*}
where $\pi_S:X\times_Z{}S\to{}S$ is the projection.
By \cite[8.1:Prop.4]{Neron_Models}, it suffices that $X$ satisfy the following
three conditions:
\begin{enumerate}
\item $X$ is quasi-compact and quasi-separated;
\item the structure map $f:X\to{}Z$ admits a section; and
\item $\tilde{f}_*(\OO_{\tilde{X}}) = \OO_{\tilde{Z}}$ where
$\tilde{f}:\tilde{X}\to\tilde{Z}$ is the base change of $f$ via any morphism
$\tilde{Z}\to{}Z$.
\end{enumerate}
These conditions, as well as the hypotheses for Groethendieck's result, are all
satisfied when \hbox{$Z=\spec\ZZ$} and \hbox{$X=\PP^n_\ZZ$}.  Therefore
$S\mapsto\Pic(\PP^n_S)/\pi_S^*\Pic(S)$ is representable by a group scheme $\cP$.
Moreover, by \cite[8.2:Thm.5]{Neron_Models}, $\cP$ is a disjoint union of
quasi-projective schemes (each corresponding to line bundles with a fixed
Hilbert polynomial).  And by \cite[8.4:Thm.3]{Neron_Models} (a restatement
of \cite[no.236:Thm.2.1]{FGA}), each of these components is in fact projective.
Moreover, since $H^2(\PP^n_k, \OO_{\PP^n_k})=0$ for any field~$k$,
by \cite[8.4:Prop.2]{Neron_Models} we have that $\cP$ is formally smooth over
$\spec\ZZ$.  Since $\cP$ is also locally of finite presentation (being locally
quasi-projective), it is in fact smooth (\cite[00TN]{Stacks}).

Now if $k$ is any field, we have $\cP(k) = \Pic(\PP^n_k) \isom \ZZ$, with
$d\in\ZZ$ corresponding to the line bundle $\OO_{\PP^n_k}(d)$.  This implies
that each irreducible component of $\cP$ has relative dimension zero over
$\spec\ZZ$, since otherwise, when $k$ is uncountable and algebraically closed,
$\cP(k)$ would also be uncountable.  Thus each component of $\cP$ is smooth of
relative dimension zero, hence \'etale over $\spec\ZZ$.  But $\spec\ZZ$ has no
nontrivial \'etale cover, so we conclude that $\cP$ is a disjoint union of
copies of $\spec\ZZ$, indexed by $\ZZ$.  That is, $\cP$ is the group scheme
$\underline{\ZZ}$ associated to the group $\ZZ$.  This means that $\cP(S)=\ZZ$
for any connected scheme $S$, with $d\in\ZZ$ representing the class of
$\OO_{\PP^n_S}(d)$ in $\Pic(\PP^n_S)/\pi_S^*\Pic(S)$.

Finally, for any scheme $S$, since $\cP(S)=\Pic(\PP^n_S)/\pi_S^*(\Pic(S))$, we
have a short exact sequence of abelian groups:
\begin{equation*}
\begin{tikzcd}
 0 \rar & \pi_S^*(\Pic(S)) \rar & \Pic(\PP^n_S) \rar & \cP(S) \rar & 0.
\end{tikzcd}
\end{equation*}
When $S$ is connected, since $\cP(S)=\ZZ$ this sequence splits, and we conclude
$\Pic(\PP^n_S) \cong \pi_S^*\Pic(S)\oplus\ZZ$, with the $\ZZ$ component generated by
the class of $\OO_{\PP^n_S}(1)$.
\end{proof}


We can now prove the desired isomorphism of functors.  This is also well known,
but with Lemma~\ref{lem:Pic-PP^n_S} in hand the proof is straightforward, so we
will present it.

\begin{theorem}
\label{thm:PGL-is-AUT}
There is a natural isomorphism between the following two functors from the
category of schemes to the category of groups:
\[
S \rightsquigarrow \PGL_n(S) \qquad \text{and} \qquad S \rightsquigarrow \Aut_S(\PP^{n-1}_S).
\]
\end{theorem}

\begin{proof}
For a scheme $S$, we will construct an isomorphism of groups $\PGL_n(S) \cong
\Aut_S(\PP^{n-1}_S)$. The reader will easily verify that the entire construction
is compatible with base change along a morphism $T \to S$, so that our
isomorphisms are natural in the scheme $S$. 

We first reduce to the case of $S$ connected.  If $S$ is a disjoint union
$S=\coprod_i S_i$ with each $S_i$ connected, then it is clear that $\PGL_n(S) =
\prod_i \PGL_n(S_i)$, since $\PGL_n(S)$ is the set of morphisms from $S$ to the
group scheme $\PGL_n$.  It is also clear that $\PP^{n-1}_S = \coprod_i
\PP^{n-1}_{S_i}$, and that any $S$-automorphism must stabilize each
$\PP^{n-1}_{S_i}$.  It follows that also $\Aut_S(\PP^{n-1}_S)=\prod_i
\Aut(\PP^{n-1}_{S_i})$.  Thus if we prove that
$\PGL_n(S_i)\isom\Aut(\PP^{n-1}_{S_i})$ for each $i$, we get that
$\PGL_n(S)\isom\Aut(\PP^{n-1}_S)$.

Henceforth we assume $S$ is connected.  To define a map
$\PGL_n(S)\to\Aut_S(\PP^{n-1}_S)$, let $(L;M)$ represent an element of
$\PGL_n(S)$, and choose an affine open cover $\setof{U_i}$ of $S$ such that
$L|_{U_i}$ is trivial, with isomorphisms
$\phi_i:L|_{U_i}\buildrel\cong\over\to\OO_S|_{U_i}$.  Let $R_i=\OO_S(U_i)$.
Then $\PP^{n-1}_{U_i}=\Proj{}R_i[x_1,\dots,x_n]$, and the matrix $\phi_i(M)$,
with entries in $R_i$, gives an automorphism of $\PP^{n-1}_{U_i}$.  On $U_i\cap
U_j$, with $R_{i,j}=\OO_S(U_i\cap U_j)$, the $R_{i,j}$-module automorphism
$\phi_j\phi_i^{-1}$ is multiplication by some element of $R_{i,j}^{\times}$, so
we see that $\phi_i(M)$ and $\phi_j(M)$ induce the same automorphism on
$\Proj{}R_{i,j}[x_1,\dots,x_n]$, and we get a well defined automorphism of
$\PP^{n-1}_S$.

To go the other way, an automorphism $\alpha\in\Aut_S(\PP^{n-1}_S)$ is specified
by the line bundle \hbox{$\cL=\alpha^*(\OO_{\PP^{n-1}_S}(1))$} and a basis of
global sections $\setof{s_i=\alpha^*(x_i)\st i=1,\dots,n}$.  By
Lemma~\ref{lem:Pic-PP^n_S}, there is a line bundle $L$ on $S$ and an integer $d$
such that $\cL\isom\pi_S^*(L)\otimes\OO_{\PP^{n-1}_S}(d)$.  Note
that $\dim_{\funfld(s)}H^0(s,\OO_{\PP^{n-1}_s}(1))=n$ for any
point $s\in{}S$.  Therefore
also
\begin{equation*}
\begin{aligned}
n &= \dim_{\funfld(\alpha(s))}H^0(\PP^{n-1}_{\alpha(s)},\OO_{\PP^{n-1}_{\alpha(s)}}(1)) \\
&= \dim_{\funfld(s)}H^0(\PP^{n-1}_s,\alpha^*\big(\OO_{\PP^{n-1}_s}(1)\big)) \\
&= \dim_{\funfld(s)}H^0(\PP^{n-1}_s,\OO_{\PP^{n-1}_s}(d)) \\ 
\end{aligned}
\end{equation*}
since $\pi_S^*(L)$ is trivial on $\PP^{n-1}_s$.  This clearly implies $d=1$.

Now we have $\cL\isom\pi_S^*(L)\otimes\OO_{\PP^{n-1}_S}(1)$, and so the global
sections $s_i$ defining $\alpha$ are linear combinations of $x_1,\dots,x_n$ with
coefficients in $\Gamma(S,L)$.  Thus there are elements $M_{i,j}\in\Gamma(S,L)$
such that \hbox{$s_i=\sum_jM_{i,j}\otimes{}x_j$}.  Let $M$ be the matrix with entries
$M_{i,j}$.  Then we map $\alpha$ to the class of $(L;M)$ in $\PGL_n(S)$.  It is
straightforward to show that this is inverse to the map given above.
\end{proof}

Our primary interest is the action of $\Aut_S(\PP_S^{n-1})$ on the $S$-points of
$\PP^{n-1}$.  Identifying the automorphism group with $G_n(S)$ as in the
Theorem, we can make the action explicit as follows.  An $S$-point is a morphism
$a\colon S\to\PP^{n-1}$.  By \cite[II.7.1]{Hartshorne_Bible}, such morphisms are
in bijective correspondence with equivalence classes of data
$(A;a_1,\ldots,a_n)$, where $A = a^*\OO_{\PP^{n-1}}(1)$ is a line bundle on $S$,
and $a_i = a^*(x_i)$ are global sections of $A$ that generate $A$. A second set
of data $(A;a_1', \ldots, a_n')$ determines the same morphism if and only if
there is an $\OO_S$-isomorphism $\varphi \colon A \to A'$ such that
$\varphi(a_i) = a_i'$ for $i = 1, \ldots, n$.  Then tracing through the map
$G_n(S)\to\Aut_S(\PP_S^{n-1})$ defined in the proof of
Theorem~\ref{thm:PGL-is-AUT} gives the following description.

\begin{corollary}
\label{cor:point-action}
Let $\sigma \in G_n(S) \cong \PGL_n(S)$ be given by the data $(L; M)$. Then
$\sigma$ acts on a point \hbox{$a=(A;a_1,\dots,a_n)\in\PP^{n-1}(S)$} via
\[
\sigma(a) = \left(L \otimes A; M\star\mat{a_1\\\vdots\\a_n} \right),
\]
where the column vector $M\star\mat{a_1\\\vdots\\a_n}$ can be viewed as $n$
global sections of $L \otimes A$ that generate $L \otimes A$. 
\end{corollary}

\begin{remark}
This shows how to find an example of an $S$ such that $\Aut_S(\PP^{n-1})$ does not
act transitively on $\PP^{n-1}(S)$.  Indeed, a consequence of
Corollary~\ref{cor:point-action} is that if $a=(A;a_1,\dots,a_n)$ and
$b=(B;b_1,\dots,b_n)$ are in the same orbit under $\Aut_S(\PP^{n-1})$, then
$A\otimes{}B^{-1}$ must represent an $n$-torsion class in $\Pic(S)$.  For
example, let $K$ be a number field with class number divisible by 3, and let
$A\subset{}R:=\mathcal{O}_K$ be an ideal representing an element of order 3 in the
class group.  Let $a_1,a_2$ be generators of $A$, so that $(A;a_1,a_2)$
represents a point $a\in\PP^1(R)$.  We also have the point $b=(B;b_1,b_2)$ with
$B=R$, $b_1=0$, and $b_2=1$.  Then the class of $A\otimes B^{-1}=A$ is not
$2$-torsion, so $A$ and $B$ are in different orbits.
\end{remark}


\subsection{The Case of a Domain}

If $S$ is an integral scheme, then every line bundle $L$ on $S$ embeds into the
constant sheaf of rational functions $\mathcal{K}_S$. Global sections of $L$ map
to elements of the rational function field $K = \funfld(S)$. Via these
identifications, it is possible to identify $\PGL_n(S)$ with a subgroup of
\hbox{$\PGL_n(K) = \GL_n(K) / K^\times$}. We prove this after a few preliminary
results.

\begin{lemma}
\label{lem:auts-agree}
Let $X$ be a reduced scheme, and let $\alpha,\beta \colon X \to Y$ be a pair of
morphisms.  Suppose \hbox{$\alpha(x)=\beta(x)$} for every $x\in{}X$, and the induced
maps on local rings agree at every generic point of~$X$. Then $\alpha=\beta$.
\end{lemma}

\begin{proof}
Since $\alpha$ and $\beta$ agree as continuous maps, it suffices to work locally
and compare the induced morphisms of ringed structures. To that end, we may
assume that $X = \Spec B$ and $Y = \Spec A$. Then $\alpha$ and $\beta$ are
induced by ring homomorphisms $f,g \colon A \to B$, respectively. For any prime
ideal $\pp$ of $B$, we know that $f^{-1}(\pp) = g^{-1}(\pp) = :\qq$. We have the
following commutative diagram (for~$f$~or~for~$g$):
\begin{equation*}
  \begin{tikzcd}
    A
    \ar[r,"f{,}g"]
    \ar[d]
    & B
    \ar[d] \\
    A_\qq
    \ar[r,"f_{\pp}{,}g_{\pp}"]
    & B_\pp.
  \end{tikzcd}
\end{equation*}
The hypothesis that $\alpha$ and $\beta$ agree on generic points shows that
$f_\pp = g_\pp$ when $\pp$ is a minimal prime ideal. Suppose that $a \in A$, and
consider the element $b = f(a) - g(a)$. The above diagram shows that $b$
vanishes in $B_\pp$ for every minimal prime ideal $\pp$ of $B$.  In particular
this means that $b$ is contained in $\pp$.  Since every prime ideal contains a
minimal prime (using Zorn's Lemma), it follows that $b$ lies in every prime
ideal of $B$, so that $b$ is nilpotent. As $X$ is reduced, we conclude that $b =
0$. Hence $\alpha$ and $\beta$ agree.
\end{proof}

\begin{remark}
The hypothesis that $X$ is reduced is essential in Lemma~\ref{lem:auts-agree}.
For example, take
\begin{equation*}
X=\spec\ZZ_p[t]/(t^2,pt).
\end{equation*}
Then the automorphism induced by
$t\mapsto-t$ is the identity map on points, and induces the identity map on the
local ring at the generic point $(t)$, but is clearly not the identity map on
$X$.
\end{remark}

\begin{lemma}
\label{lem:auts-on-separated}
Let $S$ be an irreducible scheme with generic point $\eta\in{S}$, and let $X$ be
an irreducible $S$-scheme with separated structure map $\pi:X\to{}S$.  Given an
automorphism $\alpha\in\Aut_S(X)$, let $\alpha_\eta$ be the restriction of
$\alpha$ to $X_\eta=\pi^{-1}(\eta)$.  If $\alpha_\eta$ is the identity morphism,
then $\alpha(x)=x$ for all $x\in{}X$.
\end{lemma}

\begin{proof}
Let $\Gamma\subset{}X\times_SX$ be the graph of $\alpha$ --- \emph{i.e.}, $\Gamma$ is
the image of the morphism $1\times\alpha: X\to{}X\times_S{}X$.  Also, let
$\Delta$ be the diagonal morphism.  Suppose we knew that $\Gamma=\Delta(X)$.  Then
for any $x\in{}X$ we would have $(1\times\alpha)(x)=\Delta(x')$ for some
$x'\in{}X$.  Then applying the first projection gives $x=x'$, and applying the
second projection gives $\alpha(x)=x'$, so we would get $\alpha(x)=x$.  Thus it
suffices to prove that $\Gamma=\Delta(X)$.

Note that $1\times\alpha = (1,\alpha)\circ\Delta$, where $(1,\alpha)$ is the map
from $X\times_SX$ to itself that is the identity on the first factor and
$\alpha$ on the second.  Evidently $(1,\alpha)$ is an automorphism, as it has
the inverse $(1,\alpha^{-1})$.  Since $X$ is separated over $S$, $\Delta$ is a
closed immersion, and we find that $1\times\alpha$ is also a closed immersion.
Thus $\Gamma$ is the closure of $(1\times\alpha)(\xi)$ where $\xi$ is the
generic point of $X$.  Since $\Delta(X)$ is also the closure of $\Delta(\xi)$,
it suffices to prove that $(1\times\alpha)(\xi) = \Delta(\xi)$.  But on $X_\eta$
we have \hbox{$1\times\alpha_\eta = \Delta_{X_\eta}$, so $(1\times\alpha)(\xi) =
(1\times\alpha_\eta)(\xi) = \Delta_\eta(\xi) = \Delta(\xi)$}.
\end{proof}

\begin{proposition}
\label{prop:aut-embeds-in-generic}
Let $S$ be an integral scheme with field of fractions $K = \funfld(S)$. Let $X$
be a reduced and separated $S$-scheme such that every irreducible component of
$X$ dominates $S$. Then there is a natural inclusion $\Aut_S(X) \hookrightarrow
\Aut_K(X_K)$, where $X_K$ is the generic fiber of $X$.
\end{proposition}

\begin{proof}
The restriction of an automorphism of $X \to S$ to the generic fiber is an
automorphism, so we obtain a natural homomorphism $\Aut_S(X) \rightarrow
\Aut_K(X_K)$. We must show that this homomorphism is injective.

Every generic point of $X$ lives in the fiber over $\Spec K$. If $\alpha \in
\Aut_S(X)$ is trivial when restricted to the generic fiber, then $\alpha$ agrees
with the identity at all generic points of $X$.  Let $\zeta\in{}X$ be one of
those generic points, and let $Z\subset{}X$ be its closure, with the reduced
induced scheme structure.  Since the inclusion $Z\hookrightarrow{}X$ is a closed
immersion, in particular it is separated, and therefore $Z$ is separated over
$S$.  Then by Lemma~\ref{lem:auts-on-separated}, we have $\alpha(z)=z$ for every
point $z\in{}Z$.  Since every point of $X$ is in the closure of some generic
point, we find that $\alpha(x)=x$ for all $x\in{}X$.  Then
Lemma~\ref{lem:auts-agree} shows that $\alpha$ is the identity morphism.
\end{proof}

\begin{remark}
The hypothesis that $X$ be separated is necessary in both
Lemma~\ref{lem:auts-on-separated} and
Proposition~\ref{prop:aut-embeds-in-generic}.  A counterexample to both results
is provided by the affine line with doubled origin: Let $S=\AA^1$, and let
$X=X_1\cup X_2$ where $X_1=X_2=\AA^1$ and $X_1\smallsetminus\setof{0}$ is identified with
$X_2\smallsetminus\setof{0}$, and $X\to{}S$ is the obvious map.  The automorphism defined by
interchanging points of $X_1$ and $X_2$ acts as the identity on the dense open
subset $X_1\cap X_2$, but swaps the two copies of the point $0$.
\end{remark}

\begin{corollary}
Let $R$ be a domain with field of fractions $K$. There is a natural inclusion
\begin{equation*}
\PGL_n(R) \hookrightarrow \PGL_n(K) \cong \GL_n(K) / K^\times.
\end{equation*}

\end{corollary}

\begin{remark}
To make the inclusion more explicit, let $(L;M) \in G_n(\Spec R)$ represent an element
of $\PGL_n(R)$, where $L$ is a line bundle and $M$ a matrix of elements of $L$
with $\det(M)$ generating $L^{\otimes n}$. As each line bundle is isomorphic to
an $R$-submodule of $K$ \cite[Prop.~II.6.15]{Hartshorne_Bible}, we may assume
that $L \subset{}K$. Then $M \in \GL_n(K)$.
\end{remark}

When $R$ is a Dedekind ring, we can give a numerical characterization of the
image of $\PGL_n(R)$ in $\GL_n(K) / K^\times$. (Thanks to Aaron Gray for this
result in the case $n = 2$.)

\begin{proposition}
Let $R$ be a Dedekind ring with field of fractions $K$. Let $A = (A_{i,j}) \in \GL_n(K)$
represent an element of $\PGL_n(K)$. Then $A$ lies in the image of the
homomorphism $\PGL_n(R) \hookrightarrow \PGL_n(K)$ if and only if for each
nonzero prime ideal $\pp$ of $R$, the following holds:
\[
n \mid \ord_\pp\left(\det(A)\right) \quad \text{and} \quad
\ord_\pp(A_{i,j}) \geq \frac{1}{n} \ord_\pp\left(\det(A)\right) \text{ for all $i,j$}.
\]
\end{proposition}

\begin{proof}
Identify $\PGL_n(R)$ with $G_n(\Spec R)$. Suppose first that $A$ is in the image
and fix a nonzero prime ideal $\pp$ of $R$.  As every line bundle on $\Spec R$
is isomorphic to a nonzero ideal, we see there is an ideal $I \subset R$ and a
matrix $M = (M_{i,j})$ with entries in $I$ such that $A = xM$ for some $x \in
K^\times$. Then $\det(A) = x^n \det(M)$. Since $I^{\otimes n} \cong I^n$ is
generated by $\det(M)$, we see that
\begin{equation}
  \label{eq:dets}
  \ord_\pp\left(\det(A)\right) = n \big(\ord_\pp(x) + \ord_\pp(I)\big).
\end{equation}
Hence, $n \mid \ord_\pp\left(\det(A)\right)$. As each entry $M_{i,j}$ lies in
$I$, we have $\ord_\pp(M_{i,j}) \geq \ord_\pp(I)$. Thus, by \eqref{eq:dets}, we have
\[
\ord_\pp(A_{i,j}) = \ord_\pp(x) + \ord_\pp(M_{i,j}) \geq
\frac{1}{n}\ord_\pp\left(\det(A)\right). 
\]

For the converse, suppose that for every nonzero prime ideal $\pp$ of $R$ we
know that
\[
n \mid \ord_\pp\left(\det(A)\right) \quad \text{and} \quad
\ord_\pp(A_{i,j}) \geq \frac{1}{n} \ord_\pp\left(\det(A)\right) \text{ for all $i,j$}.
\]
Set $m_\pp = \frac{1}{n} \ord_\pp\left(\det(A)\right)$, and define
\[
I = \prod_{\pp \in \Spec R} \pp^{m_\pp} \subset K^\times.
\]
Then $I$ is a fractional ideal, and $A_{i,j} \in I$ because $\ord_\pp(A_{i,j}) \geq m_\pp$
by hypothesis. Moreover, since $\ord_\pp\left(\det(A)\right) = m_\pp \cdot n =
\ord_\pp(I^n)$, it follows that $\det(A)$ generates $I^n$. We conclude that
\hbox{$(I;A) \in G_n(\Spec R)$}.
\end{proof}


\subsection{Example: Non-matrix Element of PGL}
\label{subsec:NonMatrixExample}

We close this section with the promised example of an element of $\PGL_2(R)$
that is not an element of $\GL_2(R) / R^\times$. Of course, by the exact
sequence \eqref{eq:exact}, we know that it must arise from a nontrivial element
of $\Pic(R)[2]$. The 2-torsion points of elliptic curves naturally give rise to
such line bundles.

Let $k$ be a field of characteristic different from 2, and consider the
projective elliptic curve $E/k$ with affine equation $y^2=x^3-x$. Its affine
coordinate ring
\[
R = k[x,y] / (y^2 - x^3 + x)
\]
is a Dedekind domain.

Let $L = Rx + Ry$ be an ideal of $R$, and set
\[
M=\begin{pmatrix}x & y \\ y & x^2\end{pmatrix}.
\]
Since $L^{\otimes 2} \cong L^2 = Rx$ and $\det(M) = x$, we have $(L;M) \in G_2(\Spec R) =
\PGL_2(R)$. If $(L; M)$ were induced by a matrix in $\GL_2(R)$, then $L$ would
have trivial image in $\Pic(R)$ by \eqref{eq:exact}. But the
divisor of $L$ is the point $(0,0)$, which has order 2 in $\Pic(R)$.


\section{Equalizers in the Category of Schemes}
\label{sec:Equalizers}

In this section we present a number of results about equalizers in the category
of schemes.  It is likely that all of these results are available in the
literature in one form or another.  For convenience, and for want of a suitable
reference, we state and prove them in the form that we will need.

Throughout this section, $S$ denotes a fixed scheme. All schemes should be
viewed as $\ZZ$-schemes unless specified otherwise. 

In a category $\cC$, the equalizer of two morphisms $f,g \colon X \to Y$ is a
morphism $i \colon E_{\cC}(f,g) \to X$ satisfying the following two properties:
\begin{itemize}
\item $f \circ i = g \circ i$, and 
\item If $j \colon W \to X$ is any morphism such that $f \circ j = g \circ j$,
  then $j$ factors uniquely through $i$. That is, there exists a unique morphism
  $k \colon W \to E_{\cC}(f,g)$ such that $j = i \circ k$. 
\end{itemize}
If an equalizer exists, then it is unique up to canonical isomorphism over
$X$. We will sometimes abuse terminology and refer to the source object
$E_{\cC}(f,g)$ as the equalizer of $f$ and $g$.

If $\cC$ is the category of schemes over $S$, we will write $E_S(f,g)$
for the equalizer of two morphisms $f,g$. If $\cC$ is the category of all
schemes, we write $E_\ZZ(f,g)$ for the equalizer. (This notational distinction
will turn out to be irrelevant; see Remark~\ref{rem:no_base}.)

\begin{proposition}
\label{prop:exists}
Equalizers exist in the category of $S$-schemes.  Given
$S$-morphisms \hbox{$f,g\colon{}X\to{}Y$}, the equalizer of $f$ and $g$ is given
by the fiber product $E(f,g)=X\times_{Y\times_S Y}Y$ relative to the map
$f\times{}g:X\to Y\times_S Y$ and the diagonal $\Delta:Y\to Y \times_S Y$, with
$E(f,g)\to X$ projection onto the first factor.
\end{proposition}

\begin{proof}
Write $\pi_X, \pi_Y$ for the projection morphisms onto the two factors of
$X \times_{(Y \times_S Y)} Y$.
To argue that
$\pi_X \colon X \times_{(Y \times_S Y)} Y \to X$
is the equalizer of $f$ and $g$ in the category of $S$-schemes, we must show
three things:
\begin{enumerate}
\item[(a)] $f \circ \pi_X = g \circ \pi_X$;
  \item[(b)] If $j \colon W \to X$ is an $S$-morphism such that $f \circ j =
g \circ j$, then there exists an $S$-morphism
\begin{equation*}
k{}\colon{}W\to{}X{}\times_{(Y\times_S{}Y)}{}Y
\end{equation*}
such that $j = \pi_X \circ k$;
and \item[(c)] The morphism $k$ in (b) is unique.
\end{enumerate}

For (a), write $p,q$ for the first and second projections of $Y \times_S Y$ to
$Y$. Combining the fiber product square for $X{}\times_{(Y\times_S{}Y)}{}Y$ with
the diagram defining the map $f \times g \colon X \to Y \times_S Y$, we obtain
the following commutative diagram:  
\begin{equation}
  \label{diag:exist}
    \begin{tikzcd}
      X \times_{(Y \times_S Y)} Y
      \ar[r,"\pi_Y"]
      \ar[d,"\pi_X"']
      & Y
      &
      \\ X
      \ar[r,"f \times g", near end]
      \ar[dr,"f",bend right=20]
      \ar[rr,"g",bend left=25]
      & Y \times_S Y
      \ar[d,"p"']
      \ar[r,"q", near start]
      \ar[from=u,crossing over,"\Delta", near start]
      & Y
      \ar[d]
      \\
      & Y
      \ar[r]
      & S      
    \end{tikzcd}
\end{equation}
Using the fact that $p \circ \Delta = q \circ \Delta = 1_Y$, we can chase around
the diagram to see that
\[
f \circ \pi_X = p \circ \Delta \circ \pi_Y = 1_Y \circ \pi_Y = q\circ \Delta
\circ \pi_Y = g \circ \pi_X.
\]
This proves (a).

For (b), suppose that $j \colon W \to X$ is such that $f \circ j = g \circ
j$. Note that $j \colon W \to X$ and $f\circ j \colon W \to Y$ can be used to
define a morphism $W \to X{}\times_{(Y\times_S{}Y)}{}Y$ provided we can show that 
\begin{equation}
  \label{eq:exists}
  (f \times g) \circ j = \Delta \circ f \circ j.
\end{equation}
As the target of these morphisms is $Y \times_S Y$, it suffices to show equality
in \eqref{eq:exists} after applying the projections $p$ and $q$. Noting that $f
= p \circ (f \times g)$ and $g = q \circ (f \times g)$, we have
\begin{align*}
p \circ \Delta \circ f \circ j &= 1_Y \circ f \circ j = [p \circ (f \times g)]
\circ j \\
q \circ \Delta \circ f \circ j &= 1_Y \circ f \circ j = g \circ j = [q \circ (f
  \times g)] \circ j.
\end{align*}
We conclude that \eqref{eq:exists} holds. By the universal property for the
fiber product $X{}\times_{(Y\times_S{}Y)}{}Y$, there is a unique morphism $k
\colon W \to X{}\times_{(Y\times_S{}Y)}{}Y$ such that $\pi_X \circ k = j$ and
$\pi_Y \circ k = f \circ j$. This completes the proof of (b).

For (c), suppose that $k' \colon W \to X{}\times_{(Y\times_S{}Y)}{}Y$ is another
$S$-morphism such that $\pi_X \circ k' = j$. Then \eqref{diag:exist} and
\eqref{eq:exists} show that
\[
\Delta \circ \pi_Y \circ k' = (f \times g) \circ \pi_X \circ k' = (f \times g)
\circ j = \Delta \circ f \circ j. 
\]
Applying $p$ to the far ends of this equality gives $\pi_Y \circ k' = f \circ
j$. That is, $k'$ satisfies both equalities $\pi_X \circ k' = j$ and $\pi_Y
\circ k' = f \circ j$. In (b), we argued that there is a unique such morphism,
so $k' = k$.
\end{proof}

The next result shows that formation of the equalizer is both agnostic to the
choice of base scheme and also compatible with base extension.

\begin{proposition}
\label{prop:compatible}
  Let $\rho \colon T \to S$ be a morphism of schemes.
  \begin{enumerate}
    \item\label{prop:TtoS} If $f, g \colon X \to Y$ are morphisms of $T$-schemes, then they are
      also morphisms of $S$-schemes, and we have a canonical
      isomorphism $E_S(f,g) \cong E_T(f,g)$ over $X$.
    \item\label{prop:StoT} If $f,g \colon X \to Y$ are morphisms of $S$-schemes, then we have a
      canonical isomorphism $E_T(f_T,g_T) \cong E_S(f,g)_T$ over $X_T$.
  \end{enumerate}  
\end{proposition}

\begin{proof}
We begin with assertion \ref{prop:TtoS}. Let $i \colon E_T(f,g) \to X$ be the
equalizer of $f$ and $g$ in the category of $T$-schemes. We claim it is also the
equalizer in the category of $S$-schemes. To that end, we must show three
things:
\begin{enumerate}
\item[(a)] $f \circ i = g \circ i$;
  \item[(b)] If $j \colon W \to X$ is an $S$-morphism such that $f \circ j = g \circ
    j$, then there exists $k \colon W \to E_T(f,g)$ over $S$ such that
    $j = i \circ k$; and
    \item[(c)] The morphism $k$ in (b) is unique.
\end{enumerate}
That $i$ is the equalizer of $f$ and $g$ (in any category) immediately
implies~(a).

For (b), suppose that $j \colon W \to X$ is an $S$-morphism such that $f \circ j
= g \circ j$. Write $\rho_W \colon W \to S$ and $\rho_X \colon X \to T$ for the
structure morphisms of $W$ and $X$, respectively. To say $j$ is an $S$-morphism
is to say that the following diagram commutes:
\begin{equation*}
\begin{tikzcd}[column sep = large, row sep = large]
  W \ar[r,"j"] \ar[d,"\rho_W"'] \ar[dr,"\rho_X \circ j",dashed] & X \ar[d,"\rho_X"] \\
  S  & T \ar[l,"\rho"]
\end{tikzcd}
\end{equation*}
We endow $W$ with a $T$-scheme structure via the map $\rho_X \circ j$; in this
way, $j$ becomes a morphism of $T$ schemes. By the universal property for
$E_T(f,g)$, there is a $T$-morphism $k \colon W \to E_T(f,g)$ such that $j = i
\circ k$. Any $T$-morphism is also an $S$-morphism, so (b) is complete.

For (c), we cannot immediately lean on the universal property for $E_T(f,g)$ to
assert that $k$ is unique because the $T$-scheme structure on $W$ is not
unique. Suppose that $k' \colon W \to E_T(f,g)$ is a second $S$-morphism such
that $j = i \circ k'$. Let $\rho_E \colon E_T(f,g) \to T$ be the structure
morphism for $E_T(f,g)$. Then $k'$ is a $T$-morphism if we give $W$ the
structure map $\rho_E \circ k'$. As above, this gives a commutative diagram:
\begin{equation*}
\begin{tikzcd}[column sep = large, row sep = large]
  W
  \ar[r,"k'"]
  \ar[d,"\rho_W"']
  \ar[dr,"\rho_E \circ k'"]
  & E_T(f,g)
  \ar[d,"\rho_E"]
  \ar[r,"i"]
  & X
  \ar[d,"\rho_X"]
  \\ S
  & T
  \ar[l,"\rho"]
  \ar[r,equal]  
  & T
\end{tikzcd}
\end{equation*}
Since $j = i \circ k'$, composition across the top of this diagram shows that
the $T$-scheme structures $\rho_E \circ k'$ and $\rho_X \circ j$ agree on
$W$. The latter is the structure we used in part (b), so the uniqueness part of
the universal property for $E_T(f,g)$ applies to show that $k' = k$. This
completes the proof of the first part of the proposition.

For the second part of the proposition, let $i \colon E_S(f,g) \to X$ be the
equalizer of $f$ and $g$ in the category of $S$-schemes. After base extension to
$T$, we obtain the $T$-morphism $i_T \colon E_S(f,g)_T \to X_T$, which
immediately satisfies $f_T \circ i_T = g_T \circ i_T$. It suffices to show that
$i_T \colon E_S(f,g)_T \to X_T$ satisfies the definition of the equalizer in the
category of $T$-schemes. The proof strategy is similar to the above argument, so
we omit it. 
\end{proof}

\begin{remark}
  \label{rem:no_base}
As a consequence of the first part of Proposition~\ref{prop:compatible}, we no
longer need to make it explicit which category of schemes we use to compute the
equalizer. If $f,g \colon X \to Y$ are morphisms of $S$-schemes, we will now
write $E(f,g) \to X$ for the equalizer.
\end{remark}

The formation of the equalizer of two morphisms $f,g \colon X \to Y$ behaves
well with respect to precomposition.

\begin{proposition}
  \label{prop:precomp}
Let $i \colon E(f,g) \to X$ be the equalizer of two morphisms $f,g \colon X \to
Y$. If $h \colon Z \to X$ is any morphism, then the equalizer of $f \circ h$ and
$g \circ h$ is given by $\pi_Z \colon E(f,g) \times_X Z \to Z$.
\end{proposition}

\begin{proof}
  One proves that the map $\pi_Z \colon E(f,g) \times_X Z \to Z$ satisfies the
  universal property for the equalizer of $f \circ h$ and $g \circ h$. The
  strategy is virtually the same as in the proof of
  Proposition~\ref{prop:compatible}\eqref{prop:TtoS}, so we omit it.
\end{proof}

Let $X$ be a scheme over $S$. The universal property for the fiber product shows
that for any scheme $Y$, we have a canonical bijection
\[
\Hom_{\ZZ}(X,Y) \cong \Hom_S(X,Y_S).
\]
It turns out that this bijection preserves equalizers.

\begin{proposition}
\label{prop:which-category}
Let $X$ be a scheme over $S$, and $Y$ an arbitrary scheme. Let $f,g \colon X \to
Y$ be morphisms, and let $f',g' \colon X \to Y_S$ be the induced morphisms of
$S$-schemes. Then $E(f,g)$ is canonically isomorphic to $E(f',g')$ as schemes
over $X$.
\end{proposition}

\begin{proof}
Let $\rho_X \colon X \to S$ be the structure morphism for $X$. Write $i \colon
E(f,g) \to X$ for the equalizer of $f$ and $g$. Note that $i$ becomes an
$S$-morphism if we endow $E(f,g)$ with the $S$-scheme structure given by $\rho_X
\circ i$. We will show that $i$ also satisfies the universal property to be the
equalizer of $f'$ and $g'$.

First we show that
\begin{equation}
  \label{eq:primes}
  f' \circ i = g' \circ i.
\end{equation}
Since the target is the product $Y_S = Y \times_{\Spec \ZZ} S$, it suffices to show that
\eqref{eq:primes} holds after applying the two projections $\pi_Y$ and
$\pi_S$. Using the fact that $\pi_Y \circ f' = f$ and $\pi_Y \circ g' = g$, we have
\begin{align*}
  \pi_Y \circ f' \circ i &= f \circ i = g \circ i = \pi_Y \circ g' \circ i \\
  \pi_S \circ f' \circ i &= \rho_X \circ i= \pi_S \circ g' \circ i.
\end{align*}
Hence \eqref{eq:primes} holds.

Next, let $j \colon W \to X$ be an $S$-morphism such that $f' \circ j = g' \circ
j$. Applying $\pi_Y$ to both sides of this equality shows that $f \circ j = g
\circ j$, from which we conclude that there is a unique morphism $k \colon W \to
E(f,g)$ such that $i \circ k = j$. To conclude the proof, we must show that $k$
is an $S$-morphism with the given structure on $E(f,g)$. Write $\rho_W \colon W
\to S$ for the structure morphism for $W$. Then
\[
\rho_W = \rho_X \circ j = \rho_X \circ i \circ k,
\]
where the first equality follows from the fact that $j$ is an
$S$-morphism. Since $\rho_X \circ i$ is the structure map for $E(f,g)$, we
conclude that $k$ is an $S$-morphism, and we are finished.
\end{proof}

The next few results give various characterizations of the points of a fiber
product, and specifically of the equalizer of two morphisms.

\begin{lemma}
\label{lem:points-of-fiber-product}
Let $X,Y$ be $S$-schemes with structure morphisms
$\rho_X:X\to S$ and $\rho_Y:Y\to S$.  Then there is a bijection between the
points of $X\times_S Y$ and triples $(x,y,p)$ where $x\in{}X$ and $y\in{}Y$
satisfy $\rho_X(x)=\rho_Y(y)=:s$, and $p$ is a prime ideal of
$\funfld(x)\otimes_{\funfld(s)}\funfld(y)$.
\end{lemma}

\begin{proof}
For $z\in X\times_S Y$, set $x=\pi_X(z)$, $y=\pi_Y(z)$, and $s=\rho(z)$ where
$\rho=\rho_X\pi_X=\rho_Y\pi_Y$.  Then we associate to $z$ the triple
$(x,y,\ker\big(\funfld(x)\otimes_{\funfld(s)}\funfld(y)\to \funfld(z)\big)$.
Conversely, given $(x,y,p)$ satisfying the stated conditions, we get maps
\begin{equation*}
  \begin{tikzcd}[row sep = 0.05cm]
& \funfld(x) \ar[r] & \funfld(x)\otimes_{\funfld(s)}\funfld(y) \ar[r] &
    \funfld(p) \\
\kappa(s) \ar[ur] \ar[dr] & & & \\    
& \funfld(y) \ar[r] & \funfld(x)\otimes_{\funfld(s)}\funfld(y) \ar[r] & \funfld(p)
  \end{tikzcd}
\end{equation*}
which agree on the image of $\funfld(s)$.  These give rise to morphisms:
\begin{equation*}
\begin{aligned}
\xi:  &\spec\funfld(p) \to \spec\funfld(x) \to X \\
\eta: &\spec\funfld(p) \to \spec\funfld(y) \to Y.
\end{aligned}
\end{equation*}
The composition of either of these with the morphism to $S$ factors through
$\spec\funfld(s)$, and hence the two compositions agree.  Thus we get a morphism
$\spec\funfld(p)\to{}X\times_S{}Y$.  Its image is a single point
$z\in{}X\times_S{}Y$, which we associate to $(x,y,p)$.  It is straightforward to
show that these these maps are inverse to each other.
\end{proof}

\begin{remark}
This result is similar to the more familiar statement that for any morphism of
schemes \hbox{$f:X\to S$}, the set of points $f^{-1}(s)\subset X$ mapping to a given
point $s\in S$ is in bijection with the points of $\spec\funfld(s)\times_S{}X$.
In fact, this latter result can be used to give another proof of
Lemma~\ref{lem:points-of-fiber-product}. See also \hbox{\cite[Ex.~3.1.7]{Qing_Liu_Algebraic_Geometry}}.
\end{remark}

Since equalizers are fiber products with diagonal morphisms, we will want to use
a special property of diagonals.

\begin{definition}
A morphism of schemes $f:X\to{}Y$ is a \textbf{locally closed immersion} if for
each point $x\in{}X$, there exist neighborhoods $U$ and $V$ of $x$ and $f(x)$, 
respectively, such that $f(U)\subset V$ and $f|_U:U\to V$ is a closed immersion.
\end{definition}

\begin{lemma}
\label{lem:diagonal-is-lcimm}
If $S$ is a scheme and $Y$ is an $S$-scheme, then the diagonal
$\Delta:Y\to{}Y\times_S{}Y$ is a locally closed immersion.
\end{lemma}

\begin{proof}
By the previous lemma, for $y \in Y$, the point $\Delta(y)$ corresponds to the
triple $(y,y,p)$ for some prime ideal $p$ of $\funfld(y) \otimes_{\funfld(s)}
\funfld(y)$, where $s = \rho_Y(y)$. In particular, the image of $\Delta$ is
covered by open affine sets $V\times_U{}V=\spec{}B\otimes_A{}B$ where $U=\spec
A$ and $V=\spec B$ are affine open subsets of $S$ and $Y$, respectively.  Then
$\Delta|_V$ is given by the ring homomorphism $B\otimes_A{}B\to{}B:
b\otimes{}b'\mapsto{}bb'$.  This homomorphism is clearly surjective, hence
$\Delta|_V$ is a closed immersion.
\end{proof}

Locally closed immersions have the following useful properties.

\begin{lemma}
\label{lem:residue-fields-preserved}
Let $f: X\to Y$ be a locally closed immersion.  Then $f$ preserves residue
fields: For each point $x\in X$, $\funfld(x)=\funfld(f(x))$.
\end{lemma}

\begin{proof}
Let $y=f(x)$. The statement is local, so we may assume $X = \Spec A$, $Y = \Spec
B$, and the corresponding homomorphism $B \to A$ is surjective with kernel
$I$. Now $y$ is a prime ideal of $B$ containing $I$, and $x = y / I$, so that
$A/x = B/y$.  But $\funfld(x)$ is the fraction field of $A/x$ and $\funfld(y)$
is the fraction field of $B/y$, which therefore coincide.
\end{proof}

\begin{convention}
For a scheme $X$, let us write $|X|$ for the underlying topological space. If $f
\colon X \to Y$ is morphism of schemes, we write $|f| \colon |X| \to |Y|$ for
the underlying continuous map.
\end{convention}

\begin{lemma}
\label{lem:fiber-product-with-lcimm}
Let $f:X\to Z$ and $g:Y\to{}Z$ be morphisms of schemes, and assume $g$ is a
locally closed immersion.  Then the canonical continuous map
$|X\times_Z{}Y|\to|X|\times_{|Z|}|Y|$ given by $q \mapsto (\pi_X(q),\pi_Y(q))$
is a bijection.
\end{lemma}

\begin{proof}
By Lemma~\ref{lem:points-of-fiber-product}, the points of $X\times_Z{}Y$ are
given by triples $(x,y,p)$ with $f(x)=g(y)$ and
$p$ a prime ideal in $\funfld(x)\otimes_{\funfld(z)}\funfld(y)$, where $z=g(y)$.
Since $g$ is a locally closed immersion, by
Lemma~\ref{lem:residue-fields-preserved}, we have $\funfld(y)=\funfld(z)$, so
that $\funfld(x)\otimes_{\funfld(z)}\funfld(y)=\funfld(x)$, and we must have
$p=0$. It follows that the correspondence $p \mapsto (\pi_X(p), \pi_Y(p))$ is a
bijection. The universal property for the topological fiber products shows
that this map is continuous.
\end{proof}

\begin{remark}
Though we will not need this, the map in
Lemma~\ref{lem:fiber-product-with-lcimm} is actually a homeomorphism.  In fact,
$|X\times_Z{}Y|$ has a basis of open sets of the form $U\times_W{}V$ where
$U=\spec{A}$, $V=\spec{B}$ and $W=\spec{C}$ are affine open subsets of $X,Y,Z$,
respectively, and where $A$ and $B$ are $C$-algebras.  Moreover, the fact that
$Y\to{}Z$ is a locally closed immersion means that we can take $B=C/I$ for some
ideal $I$ of $C$.  Then $U\times_W{}V = \spec{}A/AI$.  Using this, one can show
that the image of $|U\times_W{}V|$ in $|X|\times_{|Z|}|Y|$ is equal to
$|U|\times_{|W|}|V|$, which is open in $|X|\times_{|Z|}|Y|$.  Thus the map is
open as well as being a continuous bijection, and hence it is a homeomorphism.
\end{remark}

Now we can describe the points of the equalizer.

\begin{theorem}
  \label{thm:equalizer-as-set}
  Let $i \colon E(f,g) \to X$ be the equalizer of two morphisms $f, g \colon X
  \to Y$. Then the set map $|i|$ is injective with image equal to
  \[
  \setof{x\in{}X\st f(x)=g(x)=:y \text{ and the maps $\funfld(y) \to \funfld(x)$
      induced by $f$ and $g$ agree} }.
  \]
\end{theorem}

\begin{proof}
By Proposition~\ref{prop:exists}, we may identify the equalizer
with $\pi_X \colon X \times_{(Y \times_S Y)} Y \to X$. Since $\Delta \colon Y \to
  Y \times_S Y$ is a
  locally closed immersion, by Lemma~\ref{lem:fiber-product-with-lcimm} we have
  a continuous bijection
  \[
  |X \times_{(Y \times_S Y)} Y| \stackrel{j}{\rightarrow} |X| \times_{|Y\times_S
    Y|} |Y|
  \]
  induced by the diagram
  \begin{equation}
    \label{eq:absdiagram}
    \begin{tikzcd}[column sep=small]
      {}|X \times_{(Y \times_S Y)} Y{}|
      \ar[drr, bend left=25, "|\pi_Y|"]
      \ar[ddr, bend right=25, "|\pi_X|"']
      \ar[dr,dashed,"j"]
      &  & \\
      &  {}|X\mid \times_{{}|Y \times_S Y{}|} {}|Y{}|
      \ar[r,"\pi_{|Y|}"]
      \ar[d,"\pi_{|X|}"]
      & {}|Y{}|
      \ar[d,"|\Delta|"] \\
      & {}| X {}|
      \ar[r,"|f \times g|"]
      & {}| Y \times_S Y{}|
    \end{tikzcd}
  \end{equation}

First we show that $\pi_X$ is injective on points. Since $j$ is a bijection, it
suffices to show that $\pi_{|X|}$ is injective. Suppose that $(x_i,y_i) \in |X|
\times_{|Y \times_S Y|} |Y|$ for $i = 1,2$ have the same image under
$\pi_{|X|}$. Then $x_1 = x_2 =:x$. Going around the other way in diagram
\eqref{eq:absdiagram} shows that
\[
|\Delta| \circ \pi_{|Y|}(x,y_1) = |\Delta| \circ \pi_{|Y|}(x,y_2).
\]
That is, $\Delta(y_1) = \Delta(y_2)$. Since $\Delta$ followed by either
projection to $Y$ is the identity, we conclude that $y_1 = y_2$. Hence
$\pi_{|X|}$ is injective.

Now we describe the image of $i = \pi_X$. Suppose that $x \in \im(i)$. Let $h
\colon \Spec \funfld(x) \to X$ be the canonical morphism. Since $i$ is a locally
closed immersion, it is an isomorphism on residue fields, and we find that
$E(f,g) \times_X \Spec \funfld(x) \cong \Spec \funfld(x)$. It follows from
Proposition~\ref{prop:precomp} that $E(f\circ h, g \circ h) \cong \Spec
\funfld(x)$. Hence, $E(f \circ h, g \circ h) \to \Spec \funfld(x)$ is the
identity, and we conclude that $f \circ h = g \circ h$. That is, $f(x)
= g(x)$ and the maps $\funfld(y) \to \funfld(x)$ induced by $f$ and $g$ agree.

Conversely, suppose that $x \in X$ is such that $f(x) = g(x)$, and let us write
$y$ for this common value. We suppose further that the maps $\varphi, \gamma
\colon \funfld(y) \to \funfld(x)$ induced by $f$ and $g$ agree. As above, write
\hbox{$h \colon \Spec \funfld(x) \to X$}. Since the fiber $i^{-1}(x)$ is
homeomorphic to $E(f, g) \times_X \Spec \funfld(x)$, and since the latter is
isomorphic to $E(f \circ h, g \circ h)$ by Proposition~\ref{prop:precomp}, it
suffices to show that $E(f \circ h, g \circ h)$ is nonempty. To that end, we may
replace $X$ with $\Spec \funfld(x)$ and $h$ with the identity. As we are
assuming that $f(x) = g(x) = y$, the question is now local on $S$ and $Y$, so we
may assume $S = \Spec R$ and $Y = \Spec A$. Then $f$ and $g$ correspond to
$R$-algebra homomorphisms $\varphi, \psi \colon A \to \funfld(x)$,
respectively. Since $f(x) = g(x)$, these two homomorphisms have the same kernel
$\mm$. Moreover, the corresponding homomorphisms $A / \mm \to \funfld(x)$ agree,
so we conclude that $\varphi = \psi$. That is, $f = g$, and hence $E(f,g)$ is
nonempty, as desired.
\end{proof}

\begin{example}
Let $k$ be a field and set $X = \Spec k[\varepsilon]/(\varepsilon^2)$.  Choose two
elements $a,b\in{}k$ with $a\ne b$ and define morphisms $f, g \colon X \to
\Aff^1_k = \Spec k[t]$ by $t \mapsto a\varepsilon$ and $t \mapsto b\varepsilon$,
respectively. Then we claim that the equalizer is $E(f,g) = \Spec k$, given as a
closed subscheme of $X$.  To see this, first observe that the unique point $x
\in X$ maps to the origin $0 \in \Aff^1_k$, and the induced maps of residue
fields $k = \funfld(0) \to \funfld(x) = k$ are the identity. Hence, the
theorem shows that $E(f,g)$ consists of a single point. We also know that
$E(f,g)$ is a locally closed subscheme of $X$, so it is either $\Spec k$ or
$\Spec k[\varepsilon]/(\varepsilon^2)$. It is evidently not the latter because
$f$ and $g$ disagree as morphisms of schemes.
\end{example}

\begin{corollary}
  \label{cor:equalizer-as-set}
  Let $X$ be an $S$-scheme and $f,g \in X(S) = \Hom_S(S,X)$. Then the equalizer
  $i \colon E(f,g) \to S$ is injective as a map of sets, and identifies
  $|E(f,g)|$ with the set $\setof{s\in{}S\st f(s)=g(s)}$.
\end{corollary}

\begin{proof}
Write $\rho \colon X \to S$ for the structure morphism of $X$. Since $f$ is an
$S$-morphism, the composition $\rho \circ f$ is the identity. For $s \in S$,
this yields a canonical isomorphism $\funfld(f(s)) \stackrel{\sim}{\to} \funfld(s)$. In
particular, if $g(s) = f(s) = x$, then the induced isomorphisms of residue
fields $\funfld(x) \to \funfld(s)$ agree. Now apply the theorem.
\end{proof}

Finally, we close with an intuitive way to characterize the equalizer, at least in
the case where the target is separated.

\begin{proposition}
  \label{prop:ci}
Let $f,g \colon X \to Y$ be morphisms of $S$-schemes. If $Y$ is separated over
$S$, then the equalizer $E(f,g) \to X$ is a closed immersion. In particular, the
equalizer is the largest closed subscheme of $X$ on which $f$ and $g$ agree.
\end{proposition}

\begin{proof}
By Proposition~\ref{prop:exists} the equalizer $E(f,g) \to X$ is a base
extension of the diagonal \hbox{$\Delta \colon Y \to Y \times_S Y$} via the
morphism $f \times g \colon X \to Y
\times_S Y$. Closed immersions are stable under base extension.
\end{proof}

\begin{example}
If $Y$ is not separated, then the equalizer of two morphisms $f,g \colon X \to
Y$ can fail to be a closed subscheme of $X$. For example, let $Y$ be the affine
line with doubled origin. Write $Y = Y_1 \cup Y_2$ with $Y_1 = Y_2 = \Aff^1$,
where $Y_1 \smallsetminus \{0\}$ is identified with $Y_2 \smallsetminus
\{0\}$. Let $f,g \colon \Aff^1 \to Y$ be the open immersions of $Y_1$ and $Y_2$,
respectively. Then the set of points on which $f$ and $g$ agree is $\Aff^1
\smallsetminus \{0\}$, which is not closed.
\end{example}


\section{Vanishing Loci for Sections of Vector Bundles}
\label{sec:ZeroSets}

Let $X$ be a scheme, and let $E$ be a vector bundle on $X$. For any global section
$s$ of $E$, we know intuitively what the zero locus of $s$ should be:
\[
\text{``zeroes of $s$'' } = \{ x \in X \ | \ s_x \in \mm_x E_x \}.
\]
However, this is only a set, and we would like to endow it with a scheme
structure. To that end, we should look for functions on $X$ that characterize
the vanishing of $s$. If we view $s$ geometrically as a map from $X$ to (the
total space of) $E$, then for all functions $f$ on $E$, the composition $f \circ
s$ is a function on $X$ that vanishes precisely when $s$ does. This motivates
the following definition.

\begin{definition}
\label{def:zero-scheme}
Given a scheme $X$, a vector bundle $E$ on $X$, and a global section $s$ of $E$,
we define the \textbf{zero scheme of $s$} to be the closed subscheme of $X$ cut
out by the ideal sheaf $\im(s^\vee)$. Here $s$ determines a homomorphism of
$\OO_X$-modules $s:\OO_X \to E$, and we take $s^\vee \colon E^\vee \to \OO_X$ to be
the dual homomorphism.
\end{definition}

\begin{remark}
Global sections of $E$ can be identified with $X$-morphisms from $X$ to the
geometric vector bundle $\operatorname{\bf{}Spec}(\Sym_X E^\vee)$
associated to $E$ (see \cite[II:ex.5.18]{Hartshorne_Bible}).  If $z$ is the zero
section, then we could also have defined the zero scheme of $s$ to be the
equalizer of $s$ and $z$, which is a closed immersion by
Proposition~\ref{prop:ci}.  In fact, Proposition~\ref{prop:universal_zero} below
will show that this agrees with our definition.
\end{remark}

The next proposition explains why this definition captures our intuition of
the zero locus of $s$. 

\begin{proposition}
\label{prop:intuitive-zeroes}
Let $X$ be a scheme, $E$ a vector bundle on $X$, and $s$ a global section
of $E$. As sets, we have
\[
Z\left(\im(s^{\vee})\right) = \{x \in X \colon s_x \in \mm_x E_x\}.
\]
\end{proposition}

\begin{proof}
It suffices to work locally, so we may assume that $X = \Spec R$, $E = R^m$, and
$s = (s_1, \ldots, s_m) \in R^m$. Then $\im(s^{\vee}) = \left\{\sum r_i s_i
\colon r_i \in R\right\}$. So
\begin{align*}
\pp \in Z(\im(s^{\vee})) &\Leftrightarrow \pp \supseteq \im(s^{\vee}) \\
&\Leftrightarrow s_i \in \pp \text{ for all } i \\
&\Leftrightarrow s_{\pp} \in \pp E_{\pp}. \qedhere
\end{align*}
\end{proof}

\begin{remark}  Definition~\ref{def:zero-scheme} is meaningful---but not
correct---for sections of general sheaves.  Indeed, the intuitive notion of zero
locus of a section may not even be closed.  For example, let $X=\Spec\ZZ_p$,
$E=\FF_p$, and $s\in{}E$ any nonzero element.  Then at the generic point
$x=(0)$, we have $E_x=0$, so $s_x\in\mathfrak{m}_xE_x$; but at the closed point
$x=(p)$, we have $\mathfrak{m}_x=0$ and $s_x=s\ne0$, so
$s_x\notin\mathfrak{m}_xE_x$.  Thus the zero locus of $s$ is $X-\setof{(p)}$,
which is open and not closed. The zero scheme, as given by
Definition~\ref{def:zero-scheme}, is $Z(\im(s^\vee)) = X$ since $E^\vee = 0$.
\end{remark}

In the case of line bundles, we have an alternate characterization in terms of
annihilators. Though this is interesting in its own right, we do not use it in
what follows.

\begin{proposition}
\label{prop:annihilator}
Let $X$ be a scheme, $L$ a line bundle on $X$, and $s$ a global section of
$L$. We have an equality of ideal sheaves
\[
\im(s^{\vee}) = \Ann\left(L/\langle s \rangle\right),
\]
where $s^{\vee} \colon L^{\vee} \to \OO_X$ is the dual to the morphism $s \colon
\OO_X \to L$.
\end{proposition}

\begin{proof}
It suffices to work locally, in which case we take $X = \Spec R$ and $L =
R$. Then $\Ann(L / \langle s \rangle ) = (s)$. An element of $L^{\vee}$ is given
by multiplication-by-$r$ for some $r \in R$. So 
\[
\im(s^{\vee}) = \{rs \colon r \in R\} = (s). \qedhere
\]
\end{proof}

\begin{remark}
The same ``zero locus'' description does not apply to $\Ann(E/\langle s\rangle)$
for a vector bundle $E$ of rank larger than~1. For example, if $X = \Spec R$, $E
= R \oplus R$, and $s = (1,0)$, we see that the map $s^\vee \colon E^\vee \to R$
is given by dot product with $s$. Hence, 
\begin{eqnarray*}
\im(s^{\vee}) = R &\Longrightarrow& Z\left(\im(s^{\vee})\right) = \varnothing \\
\Ann\left(E/\langle s \rangle\right) = 0 &\Longrightarrow&
Z\left(\Ann\left(E/\langle s \rangle\right)\right) = X.
\end{eqnarray*}
So the annihilator is not capturing the fact that $s$ is nowhere vanishing.
\end{remark}

To conclude this section, we show that the zero scheme of a section of a vector
bundle is characterized by a universal property.

\begin{proposition}
\label{prop:universal_zero}
Let $X$ be a scheme, $E$ a vector bundle on $X$, and $s$ a global section of
$E$. Write $Z$ for the zero scheme of the section $s$, and $i:Z \to X$ for the
canonical closed immersion.  Then
\begin{enumerate}
\item\label{zero1}  $i^*(s) = 0;$
\item\label{zero2} If $j: W \to X$ is a morphism such that  $j^*(s) = 0 \in j^*(E)$, then $j$
   factors uniquely through~$i$.
\end{enumerate}
\end{proposition}

\begin{proof}  
The first statement is local so it suffices to take $X$ affine, $X=\Spec R$, $E
= R^m$, and $s = (s_1, \ldots, s_m)$. By the proof of
Proposition~\ref{prop:intuitive-zeroes}, the zero scheme of $s$ is cut out by
the ideal
\[
I = (s_1, \ldots, s_m) \subset R,
\]
so $Z = \Spec R/I$. It follows that $i^*E = (R/I)^m$ and $i^*(s) = (0,\ldots,
0)$, proving \eqref{zero1}.  For \eqref{zero2}, by taking suitable open covers,
we can reduce to the case of $W$ and $X$ affine and $E$ free.  Say $X=\Spec R$
and $W=\Spec T$.  Then $s = (s_1, \ldots, s_m)$ is an element of $R^m$, and $Z =
\Spec R/I$ as above. Now $j$ gives a ring homomorphism $\varphi: R \to T$, and
$j^*(s)=0$ becomes $\varphi(s_k) = 0$ for $k = 1, \ldots, m$, so that $\varphi$
factors uniquely through $R/I$. That is, $j$ factors through $i$, at least
locally.  But by the uniqueness of factoring through a quotient ring, these
locally defined maps $W \to Z$ must agree on intersections, hence define a
global morphism.
\end{proof}

As a final remark, we note that everything in this section can be generalized to
the case of \textit{several} sections of a vector bundle $E$. If $s_0, \ldots,
s_n$ are global sections of $E$, then the common zero scheme of these sections
is the closed subscheme associated with the ideal sheaf $\im(s_0^\vee \oplus
\cdots \oplus s_n^\vee)$, where $s_0^\vee \oplus \cdots \oplus s_n^\vee$ is the
dual of the morphism of sheaves $\OO_X \to \bigoplus_{i=0}^n E$ given by $1
\mapsto (s_0, \ldots, s_n)$. Analogues of
Propositions~\ref{prop:intuitive-zeroes} and~\ref{prop:universal_zero} hold.


\section{Strongly Distinct Morphisms}
\label{sec:StrongDistinctness}

\begin{definition}
  \label{def:sd}
Let $X$ and $Y$ be schemes. We say that two morphisms $f,g \colon X \to Y$
are \textbf{strongly distinct} if the equalizer of $f$ and $g$ is
$\varnothing \to X$. (Here $\varnothing$ is the empty scheme, which is the
initial object in the category of schemes.)  We write $f \sd g$ to denote that
$f$ and $g$ are strongly distinct.
\end{definition}

A consequence of Theorem~\ref{thm:equalizer-as-set} is that two morphisms
$f, g \colon X \to Y$ are strongly distinct if they ``disagree
at every point of $X$'':

\begin{corollary}
Let $f, g \colon X \to Y$ be morphisms of schemes. If $f(x)\neq g(x)$ for all $x
\in X$, then $f$ and $g$ are strongly distinct. 
\end{corollary}

\begin{remark}
The converse of this corollary does not hold in general. For example, consider
the two morphisms $f,g \colon \Spec \QQ(i) \to \Spec \ZZ[t] = \Aff^1_\ZZ$ given
by $t \mapsto \pm i$. Then $f$ and $g$ have the same image, namely the prime
ideal $(t^2 + 1)$.  But in fact they are strongly distinct.  To see this,
consider the induced $\QQ(i)$ morphisms
$f',g' \colon\Spec \QQ(i) \to \Spec \QQ(i)[t] = \Aff^1_{\QQ(i)}$ By
Proposition~\ref{prop:which-category}, these have the same equalizer as $f$ and
$g$.  But now $f'$ and $g'$ have distinct images in $\Aff^1_{\QQ(i)}$, so
$E(f,g) = E(f',g') = \varnothing$.
\end{remark}

By Corollary~\ref{cor:equalizer-as-set}, the converse of the previous corollary
does hold if we restrict our attention to $S$-points of an $S$-scheme:

\begin{corollary}
  \label{cor:moral_sd}
Let $X$ be a scheme over $S$. Two points $f,g \in X(S) = \Hom_S(S,X)$ are
strongly distinct if and only if $f(s)\neq g(s)$ for all $s \in S$.
\end{corollary}

\begin{example}
Let $X$ be a scheme of finite type over a field $k$ --- \emph{i.e.}, a variety. Two
points $f,g \in X(k)$ are strongly distinct if and only if they are distinct in
the usual sense.
\end{example}

The property of being strongly distinct is preserved under precomposition by an
arbitrary morphism. 

\begin{proposition}
\label{prop:sd-preserved}
Let $X,Y,Z$ be schemes, and let $h \colon Z \to X$ be a morphism. The induced
map
\begin{equation*}
h^*:\Hom(X,Y) \to \Hom(Z,Y)
\end{equation*}
preserves strong distinctness. 
\end{proposition}

\begin{proof}
Proposition~\ref{prop:precomp} shows that the equalizer of $f \circ h$ and $g
\circ h$ satisfies $E(f\circ h,g\circ h) \cong E(f,g) \times_X Z$. Since $f$ and
$g$ are strongly distinct, the first projection gives a morphism $E(f\circ
h,g\circ h) \to E(f,g) = \varnothing$. It follows that $E(f \circ h, g \circ h)
= \varnothing$ since that is the only scheme that admits a morphism to the empty
scheme. (A map $0 \to R$ is a ring homomorphism if and only if $R = 0$.) Hence
$f \circ h$ and $g \circ h$ are strongly distinct, as desired.
\end{proof}

\begin{remark} 
Definition~\ref{def:sd} can be generalized in an obvious way to talk about
strongly distinct morphisms in a category with an initial object
$\varnothing$. If we insist that the initial object is \textit{strictly
  initial} --- \emph{i.e.}, any morphism $X \to \varnothing$ is an isomorphism ---
then Proposition~\ref{prop:sd-preserved} and its proof are valid in this more general
setting.
\end{remark}

We now give an alternate description of the equalizer in the case of morphisms
to $\PP^1$.

\begin{theorem}
\label{thm:equalizer}
Let $a,b: S \to \PP^1$ be morphisms determined by line bundles $A$ and $B$ with
global sections $(a_0,a_1)$ and $(b_0,b_1),$ respectively.  Let $Z \subset S$ be
the zero scheme of the section $a_0\otimes{}b_1-a_1\otimes{}b_0$ of $A \otimes
B$. Then $Z\to{}S$ is the equalizer of $a$ and $b$.
\end{theorem}

We give the proof after first addressing a preliminary statement.

\begin{lemma}
\label{lem:2}
Let $Z$ be a scheme and $A, B$ be line bundles on $Z$ generated by global
sections $(a_0,a_1)$ and $(b_0,b_1)$, respectively, and suppose $a_0\otimes{}b_1
= a_1\otimes{}b_0\in A\otimes{}B$.  Then there exists a unique isomorphism $A
\to B$ mapping $a_0 \mapsto b_0$ and $a_1\mapsto b_1$.
\end{lemma}

\begin{proof}  We can find an affine open cover $\setof{U_i}$ of $Z$ such that $A$ and $B$ are
both trivial on each $U_i$.  We can further assume that on each $U_i$, either $a_0$ or
$a_1$ generates $A$.  Suppose $a_0$ generates $A$ on
$U_i$, and let $a_1 = r_i a_0$. Then
\begin{equation*}
\begin{aligned}
	0 &= a_0 \otimes b_1 - a_1 \otimes{} b_0 \\
	  &= a_0 \otimes b_1 - r_i a_0 \otimes{} b_0 \\
	  &= a_0 \otimes ( b_1 - r_i b_0 ).
\end{aligned}
\end{equation*}
Now $A|_{U_i}$ is a rank-one free module with generator $a_0$, so $a_0\otimes{}u = 0$ implies
$u=0$, and we get $b_1 = r_i b_0$, from which it also follows that $b_0$ generates $B$
on $U_i$.  Then we can define an isomorphism $\varphi_i : A|_{U_i} \to B\mid_{U_i}$ by
sending $a_0$ to $b_0$, which then also maps $a_1$ to $b_1$.  If instead $a_1$ generates $A$
on $U_i,$ a symmetrical argument shows that also $b_1$ generates and there is
again an isomorphism $\varphi_i$ sending $a_0$ to $b_0$ and $a_1$ to $b_1.$  Clearly these
isomorphisms glue to form a global isomorphism $A \to B$.
\end{proof}

\begin{proof}[Proof of Theorem~\ref{thm:equalizer}] 
Let $i:Z \to{}S$ be the canonical closed immersion, where $ZS$ is the zero
scheme of $u:=a_0\otimes{}b_1-a_1\otimes{}b_0.$ By
Proposition~\ref{prop:universal_zero} and the definition of $Z$, we have $i^*(u)
= 0$.  Then by Lemma~\ref{lem:2}, there is an isomorphism $i^*(A)\to{}i^*(B)$
taking $i^*(a_0)$ to $i^*(b_0)$ and $i^*(a_1)$ to $i^*(b_1).$ It follows that
the data $\left(i^*(A);i^*(a_0),i^*(a_1)\right)$ and
$\left(i^*(B);i^*(b_0),i^*(b_1)\right)$ define the same morphism $Z \to \PP^1$.
But these morphisms are $a\circ{}i$ and $b\circ{}i$ respectively, so we find
$a\circ{}i=b\circ{}i.$

Now suppose $j: W \to S$ is such that $a\circ{}j = b\circ{}j$.  Let $x_0,x_1$ be the standard
global sections generating $\OO(1)$ on $\PP^1$.  There are isomorphisms
\begin{equation*}
\begin{aligned}
	\varphi&: a^*(\OO(1)) \to A, & &\varphi(a^*(x_0)) = a_0, & &\varphi(a^*(x_1)) = a_1 \\
	\psi&: b^*(\OO(1)) \to B, & &\psi(b^*(x_0)) = b_0, & &\psi(b^*(x_1)) = b_1
\end{aligned}
\end{equation*}
so that $a_0\otimes{}b_1-a_1\otimes{}b_0 =
(\varphi\otimes{}\psi)(a^*(x_0)\otimes{}b^*(x_1) - a^*(x_1)\otimes{}b^*(x_0))$.  But then
$\varphi$ and $\psi$ pull back to isomorphisms $j^*(\varphi)$ and $j^*(\psi)$ from
$\left(a\circ{}j\right)^*(\OO(1)) = (b\circ{}j)^*(\OO(1))$ to $j^*(A)$ and $j^*(B)$,
respectively, and we have:
\begin{equation*}
\begin{aligned}
	j^*( a_0\otimes{}b_1 - a_1\otimes{}b_0 )
	&=  (j^*(\varphi)\otimes{}j^*(\psi))\,
                \big(  (a\circ{}j)^*(x_0)\otimes{}(b\circ{}j)^*(x_1)
                 - (a\circ{}j)^*(x_1)\otimes{}(b\circ{}j)^*(x_0) \big) \\
	&=  (j^*(\varphi)\otimes{}j^*(\psi))\, \big( (a\circ{}j)^*( x_0\otimes{}x_1 - x_1\otimes{}x_0 ) \big) \\
	&=  0,
\end{aligned}
\end{equation*}
since $x_0\otimes{}x_1 = x_1\otimes{}x_0\in\OO(1)\otimes\OO(1) = \OO(2)$.  By
Proposition~\ref{prop:universal_zero}, it follows that $j$ factors uniquely through
$i:Z\to S$.  This proves the theorem.
\end{proof}

An immediate consequence of Theorem~\ref{thm:equalizer} is the following
important criterion for strong distinctness, which we will use frequently in the
sequel.

\begin{corollary}[Determinant Criterion]
\label{cor:sd-criterion}
Let $a,b \in \PP^1(S)$ be morphisms determined by line bundles $A$ and $B$ with
global sections $(a_0,a_1)$ and $(b_0,b_1),$ respectively. Then $a \sd b$ if and
only if $a_0 \otimes b_1 - a_1 \otimes b_0$ generates $A \otimes B$.  In
particular, if $a \sd b$, then $B \cong A^{-1}$.
\end{corollary}


\section{Cross-ratios}
\label{sec:CrossRatios}

Let $S$ be a fixed scheme throughout this section. 

\begin{lemma}
  \label{lem:sd3}
Let $a,b,c \colon S \to \PP^1_S$ be pairwise strongly distinct. Then $a^*\OO(1)
\cong b^*\OO(1) \cong c^*\OO(1)$, and each of these line bundles is 2-torsion in
$\Pic(S)$. 
\end{lemma}

\begin{proof}
Set $A = a^* \OO(1)$, $B = b^*\OO(1)$, and $C = c^*\OO(1)$. Since $a \sd b \sd c
\sd a$, by Corollary~\ref{cor:sd-criterion}, we have $A\cong B^{-1}\cong C\cong
A^{-1}$, and consequently also $A \cong B$.
\end{proof}

If $a,b \in \PP^1(S)$ are represented by the data $(A;a_0,a_1)$ and
$(B;b_0,b_1)$, respectively, define the section $\Delta(a,b) \in A \otimes B$ by
$a_0 \otimes b_1 - a_1 \otimes b_0$. If $a,b$ are strongly distinct, then the
Determinant Criterion shows that $\Delta(a,b)$ is a nowhere vanishing section of
$A \otimes B$. 
    
\begin{definition}[Cross-Ratio]
\label{def:cross-ratio}
Let $a,b,c,d \in \PP^1(S)$ be four pairwise strongly distinct points,
represented by the data $(L;a_0,a_1)$, $(L;b_0,b_1)$, $(L;c_0,c_1)$, and
$(L;d_0,d_1)$, respectively. (Lemma~\ref{lem:sd3} shows that we may choose all
of the line bundles to be equal.) Choose an isomorphism $\varphi \colon
L^{\otimes 2} \to \OO_S$. The \textbf{cross-ratio} of these four points, denoted
$(a,b;\,c,d)$, is the element of $\Gamma(S,\OO_S)$ defined by
\[
(a,b;\,c,d) = \frac{\varphi\big(\Delta(a,c)\big) \ \varphi\big(\Delta(b,d)\big)}
{\varphi\big(\Delta(a,d)\big) \ \varphi\big(\Delta(b,c)\big)}.
\]
\end{definition}

The following are simple but useful consequences of the definition:

\begin{proposition}
With notation as in Definition~\ref{def:cross-ratio},
\begin{enumerate}
\item $(a,b;c,d)$ is independent of the choice of isomorphism $\varphi$
\item $(a,b;c,d)$ is a unit in $\Gamma(S,\OO_S)$
\end{enumerate}
\end{proposition}

\begin{proof}
As every automorphism of $\OO_S$ is given by multiplication by a unit of
$\Gamma(S,\OO_S)$, the definition makes it clear that $(a,b;\,c,d)$ is
independent of the choice $\varphi$.  The Determinant Criterion shows that the
numerator and denominator of $(a,b;\, c,d)$ are units of $\Gamma(S,\OO_S)$, so
that $(a,b;\,c,d)$ is itself a unit.
\end{proof}

In order to see the connection with the usual cross-ratio, suppose that $S =
\Spec k$ is the spectrum of a field. Then we may take $L = \OO_{\Spec k}$ and
$\varphi \colon L \otimes L \to \OO_{\Spec k}$ is multiplication in $k$. In that
case, the cross-ratio of the four points is given by the quotient
\[
(a,b;\,c,d) =  \frac{(a_0c_1 - a_1c_0)(b_0d_1 - b_1d_0)}{(a_0d_1 - a_1d_0)(b_0c_1 - b_1c_0)}.
\]
If we dehomogenize by setting $a_1 = b_1 = c_1 = d_1 = 1$, this becomes the
familiar formula for the cross-ratio:
\[
(a,b;\,c,d) = \frac{(a - c)(b-d)}{(a-d)(b-c)}.
\]

We set the following notation for standard points of $\PP^1(S)$:
\begin{eqnarray*}
\Infinity &=& \left(\OO_S; 1,0\right) \\
\Zero &=& \left(\OO_S; 0,1\right) \\
\One &=& \left(\OO_S; 1,1\right).
\end{eqnarray*}
Note that $\Infinity$, $\Zero$, and $\One$ are pairwise strongly distinct by the
Determinant Criterion.

\begin{theorem}
\label{thm:3-transitivity}
Let $a,b,c \in \PP^1(S)$ be pairwise strongly distinct. Then there is a unique
$\sigma \in \PGL_2(S)$ such that $\sigma(a) = \Infinity$, $\sigma(b) =
\Zero$, and $\sigma(c) = \One$.
\end{theorem}

\begin{proof}
  By Lemma~\ref{lem:sd3}, we may assume that $L := a^*\OO(1) = b^*\OO(1) =
  c^*\OO(1)$, and that $L$ is 2-torsion in $\Pic(S)$. Let $x_0,x_1$ be the
  standard generators of $\OO(1)$, and let $a_i=a^*(x_i)$, $b_i=b^*(x_i)$, and
  $c_i=c^*(x_i)$.  Fix an isomorphism $\varphi:L^{\otimes2}\to\OO_S$. Define
  $\sigma$ by the pair $(L;M)$, where $M\in{}M_{2\times2}(\Gamma(L,S))$ is given
  by:
  \begin{equation}
    \label{eq:unique-sigma}
M =
\begin{pmatrix} \varphi\big(\Delta(a,c)\big) & 0 \\ 0 & \varphi\big(\Delta(c,b)\big) \end{pmatrix}
\star
\begin{pmatrix} b_1 & -b_0 \\ -a_1 & a_0 \end{pmatrix}.
  \end{equation}
Applying this to $a,b,c$ we find:
\begin{equation*}
\begin{aligned}
\sigma(a)
 &= \big(L^{\otimes2}; \varphi(\Delta(a,c))\Delta(a,b), 0\big) \\
 &= \big(\OO_X; \varphi(\Delta(a,c))\varphi(\Delta(a,b)), 0\big) \\
 &= \Infinity \\
\sigma(b)
 &= \big(L^{\otimes2}; 0, \varphi(\Delta(c,b))\Delta(a,b)\big) \\
 &= \big(\OO_X; 0, \varphi(\Delta(c,b))\varphi(\Delta(a,b))\big) \\
 &= \Zero \\
\sigma(c)
 &= \big(L^{\otimes2}; \varphi(\Delta(a,c))\Delta(c,b), \varphi(\Delta(c,b))\Delta(a,c)\big) \\
 &= \big(\OO_X; \varphi(\Delta(a,c))\varphi(\Delta(c,b)), \varphi(\Delta(c,b))\varphi(\Delta(a,c))\big) \\
 &= \One.
\end{aligned}
\end{equation*}

If $\sigma'$ is any other automorphism satisfying the desired property, then
$\tau := \sigma' \circ \sigma^{-1}$ is an automorphism fixing $\Infinity, \Zero,
\One \in \PP^1(S)$. We show that $\tau$ is the identity. Write $\tau = (T;A)$,
where $T$ is a line bundle on $S$ and $A = \mat{t_{00}&t_{01} \\ t_{10} &
  t_{11}}$ is a matrix of sections of $T$.  Then
\begin{eqnarray*}
  \tau(\Infinity) &=& \big(T; A \star \mat{1\\0} \big) = \big(T;t_{00},t_{10}\big) \\
  \tau(\Zero) &=& \big(T; A \star \mat{0\\1} \big) = \big(T;t_{01},t_{11}\big) \\
  \tau(\One) &=& \big(T; A \star \mat{1\\1} \big) = \big(T;t_{00} +
  t_{01},t_{10} + t_{11}\big).
\end{eqnarray*}
Since $\tau(\Infinity) = \Infinity$, we conclude that there is an isomorphism
$\theta \colon T \cong \OO_S$ such that $\theta(t_{00}) = 1$ and $\theta(t_{10})
= 0$. As $\tau(\Zero) = \Zero$, we may use this same isomorphism to find
$\theta(t_{01}) = 0$ and $\theta(t_{11}) = r \in
\Gamma(S,\OO_S)^\times$. Finally, since $\tau(\One) = \One$, we conclude that
\[
1 = \theta(t_{00}) + \theta(t_{01}) = \theta(t_{10}) + \theta(t_{11}) = r.
\]
Thus, $\theta$ witnesses the equivalence between $(T;A)$ and $(\OO_S;1)$; that
is, $\tau$ is the identity. 
\end{proof}

\begin{theorem}
  \label{thm:witness}
Let $a,b,c,d \in \PP^1(S)$ be pairwise strongly distinct, and let $\sigma \in
\PGL_2(S)$ be the unique automorphism such that $\sigma(a) = \Infinity$,
$\sigma(b) = \Zero$, and $\sigma(c) = \One$. Then $\sigma(d) = (\OO_S; z,1)$,
where $z = (a,b;c,d)$.
\end{theorem}

\begin{proof}
  Apply the formula for $\sigma$ given by \eqref{eq:unique-sigma} to see that
  $\sigma(d) = (\OO_S; z_0, z_1)$, where
\begin{eqnarray*}
z_0 &=& \varphi\big(\Delta(a,c)\big) \ \varphi\big(\Delta(b,d)\big) \\
z_1 &=& \varphi\big(\Delta(a,d)\big) \ \varphi\big(\Delta(b,c)\big).
\end{eqnarray*}
Since $a,b,c,d$ are pairwise strongly distinct, it follows from the Determinant
Criterion that $z_1$ is a unit in $\Gamma(S,\OO_S)$. Applying the automorphism
of $\OO_S$ given by multiplication by $z_1^{-1}$, we see that
$\sigma(d) \!=\! (\OO_S; z_0/z_1, 1)$.
Looking at the formula for the cross-ratio, we see that
$z_0 / z_1 = (a,b;\,c,d)$.   
\end{proof}

\begin{corollary}
Given two 4-tuples, $(x_1,x_2,x_3,x_4)$ and $(y_1,y_2,y_3,y_4)$, of pairwise strongly
distinct points in $\PP^1(S)$, the following are equivalent:
\begin{enumerate}
\item There exists $\gamma \in \Aut_S(\PP^1_S) \cong \PGL_2(S)$ such that
$y_i=\gamma(x_i)$ for $i=1,2,3,4$;
\item $(x_1,x_2;\,\,x_3,x_4) = (y_1,y_2;\,y_3,y_4)$.
\end{enumerate}
\end{corollary}

\begin{proof}
There is a unique automorphism $\alpha$ taking $(x_1,x_2,x_3,x_4)$ to
$(\Infinity,\Zero,\One,\mathbf{a})$ for some $\mathbf{a}\in\PP^1(S)$.  Similarly
there is a unique automorphism $\beta$ taking $(y_1,y_2,y_3,y_4)$ to
$(\Infinity,\Zero,\One,\mathbf{b})$ for some $\mathbf{b}\in\PP^1(S)$.  Thus if (1)
holds, we must have $\mathbf{a}=\mathbf{b}$.  But we also know that
$\mathbf{a}=\big(\OO_S;a,1\big)$ and $\mathbf{b}=\big(\OO_S;b,1\big)$ with
$a=(x_1,x_2;\,x_3,x_4)$ and $b=(y_1,y_2;\,y_3,y_4)$, so $\mathbf{a}=\mathbf{b}$
implies (2).  Conversely, if (2) holds, then we have $\mathbf{a}=\mathbf{b}$ and
we can set $\gamma=\beta^{-1}\alpha$.
\end{proof}

\begin{proposition}
  \label{prop:symmetries}
Let $a,b,c,d \in \PP^1(S)$ be pairwise strongly distinct. The cross-ratio
enjoys the following symmetries:
\begin{enumerate}
\item $(c,d;\,a,b) = (a,b;\,c,d)$;  
\item $(b,a;\,c,d) = (a,b;\,c,d)^{-1}$; and
\item $(a,c;\,b,d) = 1 - (a,b;\,c,d)$.  
\end{enumerate}
\end{proposition}

\begin{remark}
Note that the permutation group $S_4$ is generated by the elements $(12)$,
$(23)$, and $(13)(24)$. The three formulas in the proposition give a complete
description of how the $S_4$ action on a 4-tuple of pairwise strongly distinct
points impacts the value of its cross-ratio. In particular, it shows that the
Klein 4-group $\{1,(12)(34),(13)(24),(14)(23)\}$ stabilizes the cross-ratio.
This is well known in the classical setting.
\end{remark}

\begin{proof}[Proof of Proposition~\ref{prop:symmetries}]
  The first symmetry is immediate upon writing down the formula for $(c,d;\,a,b)$
  and flipping the signs of the four factors. 
  
  For the second symmetry, write down the formula for $(b,a;\,c,d)$; it is
  visibly the inverse of the formula for $(a,b;\,c,d)$.

  For the third symmetry, let $\sigma$ be the unique automorphism of $\PP^1_S$
  that satisfies $\sigma(a) = \Infinity$, $\sigma(b) = \Zero$, $\sigma(c) =
  \One$. By Theorem~\ref{thm:witness}, we have $(a,b;\,c,d) = z$, where
  $\sigma(d) = (\OO_S; z,1)$. Define
  \[
  \tau = \left(\OO_S; \mat{-1 & 1 \\ 0 & 1} \right) \in \PGL_2(S).
  \]
  Then
  \[
  \tau\circ\sigma(a) = \Infinity, \qquad \tau\circ\sigma(c) = \Zero, \qquad
  \tau\circ\sigma(b)
  = \One, \text{ and}
  \]
  \[
    \tau\circ \sigma(d) = \left(\OO_S; \mat{-1 & 1 \\ 0 & 1}
    \star  \mat{ z \\ 1 }\right)
    = \left(\OO_S; 1-z,1\right).
  \]
    A second application of Theorem~\ref{thm:witness} shows that
    \[
    (a,c;\,b,d) = 1-z = 1 - (a,b;\,c,d). \qedhere
    \]
\end{proof}

We close with the observation that the cross-ratio behaves well under change
of coefficients:

\begin{proposition}
Let $f \colon T \to S$ be a morphism of schemes, and let $a, b, c, d
\in \PP^1(S)$ be pairwise strongly distinct. If we write $f^*$ for the
pullback map on sections, then we have
\[
\big(f^*(a), f^*(b);\, f^*(c), f^*(d)\big)
= f^*(\,(a,b;\,c,d)\,). 
\]
\end{proposition}

\begin{proof}  
  As the automorphism in Theorem~\ref{thm:3-transitivity} is unique, we can
  build it locally and then glue. By Theorem~\ref{thm:witness}, we can compute
  the cross-ratio locally. So it suffices to assume that $S = \Spec A$, $T =
  \Spec B$, and that $f$ is given by $\psi \colon A \to B$. Moreover, we may
  assume that
  \[
    a = (\OO_S; a_0, a_1), \quad b = (\OO_S; b_0, b_1), \quad c = (\OO_S; c_0,
    c_1), \quad d = (\OO_S; d_0, d_1).
  \]
  (Here it is important to note that strong distinctness implies that all four
  line bundles are isomorphic, so we may simultaneously trivialize them.)
  Now the result is clear from a computation:
  \begin{align*}
    f^*(a,b;\,c,d) &=
    \psi \left(\frac{\Delta(a,c) \ \Delta(b,d)}{\Delta(a,d)
      \ \Delta(b,c)}\right) \\
    &= \frac{\Delta\big(\psi(a),\psi(c)\big) \ \Delta\big(\psi(b),\psi(d)\big)}
         { \Delta\big(\psi(a),\psi(d)\big) \ \Delta\big(\psi(b),\psi(c)\big)} \\
         &= \big(f^*(a), f^*(b);\, f^*(c), f^*(d)\big). \qedhere
  \end{align*}
\end{proof}


\bibliographystyle{abbrv}
\bibliography{xratios}

\end{document}